\documentclass[EJP,NODS]{ejpecpRS}
\usepackage[latin1]{inputenc}
\usepackage[sort&compress,numbers,square]{natbib}

\newcommand{\bbr}{\mathbb{R}}

\AUTHORS{%
Victor Pérez-Abreu\footnote{Department of Probability and Statistics, Center for
Research in Mathematics CIMAT, Apdo. Postal 402, Guanajuato, Gto. 36000
México, \emph{Email:} {\href{mailto:pabreu@cimat.mx}{pabreu@cimat.mx}}}
  \and  Robert Stelzer\footnote{Institute of Mathematical Finance, Ulm University,
Helmholtzstraße 18, D-89069 Ulm, Germany. \emph{Email:}
\href{mailto:robert.stelzer@uni-ulm.de}{robert.stelzer@uni-ulm.de}, \url{http://www.uni-ulm.de/mawi/finmath} }}
\TITLE{A Class of Infinitely Divisible Multivariate and Matrix Gamma Distributions
and Cone-valued Generalised Gamma Convolutions} 

\SHORTTITLE{Infinitely Divisible Multivariate Gamma Distributions}

\KEYWORDS{infinite divisibility; random matrix; cone valued
distribution; Lévy process; matrix subordinator} 

\AMSSUBJ{60E07; 60G51} 
\AMSSUBJSECONDARY{60B20;  60G60} 

\SUBMITTED{January 1st, 2012} 
\ACCEPTED{NA} 

\VOLUME{0}
\YEAR{2012}
\PAPERNUM{0}
\DOI{NA}

\ABSTRACT{Classes of multivariate and cone valued infinitely divisible Gamma
distributions are introduced. Particular emphasis is put on the cone-valued
case, due to the relevance of infinitely divisible distributions on the
positive semi-definite matrices in applications. The cone-valued class of
generalised Gamma convolutions is studied. In particular, a characterisation
in terms of an Itô-Wiener integral with respect to an infinitely divisible
random measure associated to the jumps of a Lévy process is established.

A new example of an infinitely divisible positive definite Gamma random matrix
is introduced. It has properties which make it appealing for modelling under
an infinite divisibility framework. An interesting relation of the 
moments of the Lévy measure and the Wishart distribution is highlighted which we suppose to be important when considering the limiting distribution of the eigenvalues . 
}
\begin{document}
\section{Introduction}

The classical examples of multivariate and matrix Gamma distributions in the
probability and statistics literature are not necessarily infinitely divisible
\cite{Gr84}, \cite{Le48}, \cite{VJ67}. These examples are analogous to
one-dimensional Gamma distributions and are obtained by a direct
generalisation of the one-dimensional probability densities; see for example
\cite{GuKa00}, \cite{Ma97}, \cite{Mu82}. Working in the domain of Fourier
transforms, some infinitely divisible matrix Gamma distributions have recently
been considered in \cite{BNPA08}, \cite{PAS08}. Their Lévy measures are direct
generalisations of the one-dimensional Gamma distribution. The work of
\cite{PAS08} arose in the context of random matrix models relating classical
and free infinitely divisible distributions.\smallskip

The study of infinitely divisible random elements in cones has been considered
in \cite{BPS01}, \cite{PS04}, \cite{PR07}, \cite{RS03} and references therein.
They are important in the construction and modeling of cone increasing Lévy
processes. In the particular case of infinitely divisible positive-definite
random matrices, their importance in applications has been recently
highlighted in \cite{BNSe07}, \cite{BNSe09}, \cite{PiSe09a} and \cite{PiSe09}.
This is due to the fact that infinite divisibility allows modelling by matrix
Lévy and Ornstein-Uhlenbeck processes, which are in those papers used to model
the time dynamics of a $d\times d$ covariance matrix to obtain a so-called
stochastic volatility model (for observed series of financial data).\smallskip

Generalized Gamma Convolutions (GGC) is a rich and interesting class of
one-di\-men\-sion\-al infinitely divisible distributions on the cone $\mathbb{R}%
_{+}=[0,\infty).$ It is the smallest class of infinitely divisible
distributions on $\mathbb{R}_{+}$ that contains all Gamma distributions and
that is closed under classical convolution and weak convergence. This class
was introduced by O. Thorin in \ a series of papers and further studied by L.
Bondesson in his book \cite{Bo92}. The book of Steutel and Van Harn
\cite{SteutelVanHarn2004} contains also many results and examples about GGC.
Several well known and important distributions on $\mathbb{R}_{+}$ are GGC.
The recent survey paper by James, Roynette and Yor \cite{JRY08} contains a
number of classical results and old and new examples of GGC. The multivariate
case was considered in Barndorff-Nielsen, Maejima and Sato \cite{BMS06}%
.\smallskip

There are three main purposes in this paper. We formulate and study
multivariate and cone valued Gamma distributions which are infinitely
divisible. Second, we consider and characterise the corresponding class
$GGC(K)$ of Generalised Gamma Convolutions on a finite dimensional cone $K.$
Finally, we introduce a new example of a positive definite random matrix with
infinitely divisible Gamma distribution and with explicit Lévy measure. 

The main results and organisation of the paper are as follows. Section 2
briefly presents preliminaries on notation and results about one-dimensional
GGC on $\mathbb{R}_{+}$ as well as some matrix notation. Section
\ref{sec:MGammadist} introduces a class of infinitely divisible $d$-variate
Gamma distributions $\Gamma_{d}(\alpha,\beta)$, whose Lévy measures are
analogous to the Lévy measure of the one-dimensional Gamma distribution. The
parameters $\alpha$ and $\beta$ are measures and functions on $\mathbf{S}$
(the unit sphere with respect to a prescribed norm), respectively. It is shown
that the distribution does not depend on the particular norm under
consideration. The characteristic function is derived and it is shown that the
Fourier-Laplace transform on $\mathbb{C}^{d}$ exists if $\beta$ is bounded
away from zero $\alpha-$ almost everywhere. Furthermore, the finiteness of
moments of all orders is studied and some interesting examples exhibiting
essential differences to univariate Gamma distributions are given.\smallskip

Section \ref{sec:ggc} considers cone valued Gamma distributions and their
corresponding class $GGC(K)$ of Generalised Gamma Convolutions on a cone $K,$
defined as the smallest class of distributions on $K$ which is closed under
convolution and weak convergence and contains all the so-called elementary
Gamma variables in $K$ (and also all Gamma random variables in $K$ in our new
definition). This class is characterised as the stochastic integral of a
non-random function with respect to the Poisson random measure of the jumps of
a Gamma Lévy process on the cone. This is a new representation in the
multivariate case extending the Wiener-Gamma integral characterization of
one-dimensional GGC on $\mathbb{R}_{+}=[0,\infty)$, as considered, for
example, in \cite{JRY08}.

Section \ref{sec:conegammadist} considers the special cone valued case of
infinitely divisible positive-semidefinite $d\times d$ matrix Gamma
distributions. New examples are introduced via an explicit form of their Lévy
measure. They include as particular cases the examples considered in
\cite{BNPA08}, \cite{PAS08}. A detailed study is done of the new two parameter
positive definite matrix distribution $A\Gamma(\eta,\Sigma)$, where
$\eta>(d-1)/2$ and $\Sigma$ is a $d\times d$ positive definite matrix. This
special infinitely divisible Gamma matrix distribution has several modeling
features similar to the classical (but non-infinitely divisible) matrix Gamma
distribution defined through a density, in particular the Wishart
distribution. Namely, moments of all orders exist, the matrix mean is
proportional to $\Sigma$ and the matrix of covariances equals the second
moment of the Wishart distribution. When $\Sigma$ is the $d\times d$
\ identity matrix $\mathrm{I}_{d}$, the distribution is invariant under
orthogonal conjugations and the trace of a random matrix $M$ with distribution
$A\Gamma(\eta,\mathrm{I}_{d})$ has a one-dimensional Gamma distribution. A
relation of the moments of the Marchenko-Pastur distribution with the
asymptotic moments of the Lévy measure is exhibited. Hence, this matrix Gamma
distribution has a special role when dealing with a random covariance matrix
and its time dynamics, e.g. by specifying it as a matrix Lévy or
Ornstein-Uhlenbeck process. As an application, the matrix Normal-Gamma
distribution is introduced, which is a matrix extension of the one-dimensional
variance Gamma distribution of \cite{MadanCarrChang1998} which is popular in finance.

\section{Preliminaries}

\label{sec:notation} For the general background in infinitely divisible
distributions and L\'evy processes we refer to the standard references, e.g.
\cite{sato1999}.

\subsection{One-dimensional GGC}

A positive random variable $Y$ with law $\mu=\mathcal{L}(Y)$ belongs to the
class of Generalised Gamma Convolutions (GGC) on $\mathbb{R}_{+}=[0,\infty)$,
denoted by $T(\mathbb{R}_{+}),$ if and only if there exists a positive Radon
measure $\upsilon_{\mu}$ on $(0,\infty)$ and $a>0$ such that its Laplace
transform is given by:%
\begin{equation}
\emph{L}_{\mu}(z)=\mathbb{E}e^{-zY}=\exp\left(  -az-\int_{0}^{\infty}%
\ln\left(  1+\frac{z}{s}\right)  \ \upsilon_{\mu}(\mathrm{d}s)\right)
\label{cfGGC}%
\end{equation}
with
\begin{equation}
\int_{0}^{1}|\log x|\upsilon_{\mu}(\mathrm{d}x)<\infty,\text{ }\int
_{1}^{\infty}\frac{\upsilon_{\mu}(\mathrm{d}x)}{x}<\infty. \label{tmGGC}%
\end{equation}
For convenience we shall work without the translation term, \textit{i.e.
}with\textit{ }$a=0.$ The measure $\upsilon_{\mu}$ is called the Thorin
measure of $\mu$. Its Lévy measure is concentrated on $(0,\infty)$ and is such
that:
\begin{equation}
\nu_{\mu}(\mathrm{d}x)=x^{-1}\mathit{l}_{\mu}(x)\mathrm{d}x, \label{LevMeaGGC}%
\end{equation}
where $\mathit{l}_{\mu}$ is a completely monotone function in $x>0$ given by
\begin{equation}
\mathit{l}_{\mu}(\mathrm{d}x)=\int_{0}^{\infty}e^{-xs}\upsilon_{\mu
}(\mathrm{d}s). \label{lmth}%
\end{equation}

The class $T(\mathbb{R}_{+})$ can be characterized by Wiener-Gamma
representations. Specifically, a positive random variable $Y$ belongs to
$T(\mathbb{R}_{+})$ if and only if there is a Borel function $h:\mathbb{R}%
_{+}\rightarrow\mathbb{R}_{+}$ with
\begin{equation}
\int_{0}^{\infty}\ln(1+h(t))\mathrm{d}t<\infty, \label{intcondh}%
\end{equation}
such that $Y\overset{\mathcal{L}}{=}Y^{h}$ has the Wiener-Gamma integral
representation 
\begin{equation}
Y^{h}\overset{\mathcal{L}}{=}\int_{0}^{\infty}h(u)\mathrm{d}\gamma_{u},
\label{WinGamInt}%
\end{equation}
where $\left(  \gamma_{t};t\geq0\right)  $ is the standard Gamma process with
Lévy measure $\nu($\textrm{d}$x)=e^{-x}\frac{\mathrm{d}x}{x}.$ The relation
between the Thorin function $h$ and the Thorin measure $\upsilon_{\mu}$ is as
follows: $\ \upsilon_{\mu}$ is the image of the Lebesgue measure on $(0,\infty)$
under the application : $s\rightarrow1/h(s)$. That is,
\begin{equation}
\int_{0}^{\infty}e^{-\frac{x}{h(s)}}\mathrm{d}s=\int_{0}^{\infty}%
e^{-xz}\upsilon_{\mu}(\mathrm{d}z),\quad x>0. \label{RelU&h}%
\end{equation}
On the other hand, if $F_{\upsilon_{\mu}}(x)=\int_{0}^{x}\upsilon_{\mu}%
($\textrm{d}$y)$ for $x\geq0$ and $F_{\upsilon_{\mu}}^{-1}(s)$ is the the
right continuous generalised inverse of $F_{\upsilon_{\mu}}(s)$, that is
$F_{\upsilon_{\mu}}^{-1}(s)=\inf\{t>0;F_{\upsilon_{\mu}}(t)\geq s\}$ for
$s\geq0,$ then, $h(s)=1/F_{\upsilon_{\mu}}^{-1}(s)$ for $s\geq0$.

Many well known distributions belong to $T(\mathbb{R}_{+})$. The positive
$\alpha$-stable distributions, $0<\alpha<1$, are GGC with $h(s)=\{s\theta
\Gamma(\alpha+1)\}^{-\frac{1}{\alpha}}$ for a $\theta>0.$ In particular, for
the $1/2-$stable distribution, $h(s)=4\left(  s^{2}\pi\right)  ^{-1}.$ Beta
distribution of the second kind, lognormal and Pareto are also GGC, see
\cite{JRY08}.

For more details on univariate GGCs we refer to \cite{JRY08,Bo92}

\subsection{Notation}

$\mathbb{M}_{d}(\mathbb{R})$ is the linear space of $d\times d$ matrices with
real entries and $\mathbb{S}_{d}$ its subspace of symmetric matrices. By
$\mathbb{S}_{d}^{+}$ and $\overline{\mathbb{S}}_{d}^{+}$ we denote the open
(in $\mathbb{S}_{d}$) and closed cones of positive and nonnegative definite
matrices in $\mathbb{M}_{d}(\mathbb{R})$. $\mathbf{S}_{\mathbb{R}^{d}%
,\Vert\cdot\Vert}$ is the unit sphere on $\mathbb{R}^{d}$ with respect to the
norm $\Vert\cdot\Vert.$

The Fourier transform $\widehat{\mu}$ of a measure $\mu$ on $\mathbb{M=R}^{d}$
or $\mathbb{M}=\mathbb{M}_{d}(\mathbb{R})$ is given by
\[
\hat{\mu}(z)=\int_{\mathbb{M}}e^{i\left\langle z,x\right\rangle }%
\mu(\mathrm{d}x)\quad z\in\mathbb{M}%
\]
where we use $\left<  A,B\right>  =\mathrm{tr}(A^{\top}B)$ as the scalar
product in the matrix case, where $A^{\top}$ denotes the transposed on
$\mathbb{M}_{d}(\mathbb{R})$. By $\mathrm{I}_{d}$ we denote the $d\times d$
identity matrix and by $\left\vert A\right\vert $ the determinant of a square
matrix $A$. For a matrix $A$ in the linear group $\mathcal{GL}_{d}%
(\mathbb{R})$ we write $A^{-\top}=\left(  A^{\top}\right)  ^{-1}.$

We say that the distribution of a symmetric random $d\times d$ matrix $M$ is
invariant under orthogonal conjugations if the distribution of
$OMO^{\mathrm{\top}}$ equals the distribution of $M$ for any non-random matrix
$O$ in the orthogonal group $\mathcal{O}(\mathrm{d})$. Note that $M\rightarrow
OMO^{\mathrm{\top}}$ with $O\in$ $\mathcal{O}(\mathrm{d})$ are all linear
orthogonal maps on $\mathbb{S}_{d}^{+}$ (or $\mathbb{M}_{d}$) preserving
$\mathbb{S}_{d}^{+}$.

\section{Multivariate Gamma Distributions}

\label{sec:MGammadist}

\subsection{Definition}

\begin{definition}
\label{def:mgamma} Let $\mu$ be an infinitely divisible probability
distribution on $\mathbb{R}^{d}$. If there exists a finite measure $\alpha$ on
the unit sphere $\mathbf{S}_{\mathbb{R}^{d},\Vert\cdot\Vert}$ with respect to
the norm $\Vert\cdot\Vert$ equipped with the Borel $\sigma$-algebra and a
Borel-measurable function $\beta:\mathbf{S}_{\mathbb{R}^{d},\Vert\cdot\Vert
}\rightarrow\mathbb{R}_{+}$ such that
\begin{equation}
\hat{\mu}(z)=\exp\left(  \int_{\mathbf{S}_{\mathbb{R}^{d},\Vert\cdot\Vert}%
}\int_{\mathbb{R}_{+}}\left(  e^{irv^{\top}z}-1\right)  \frac{e^{-\beta(v)r}%
}{r}\mathrm{d}r\alpha(\mathrm{d}v)\right)  \label{FTGMult}%
\end{equation}
for all $z\in\mathbb{R}^{d}$, then $\mu$ is called a \emph{$d$-dimensional
Gamma distribution} with parameters $\alpha$ and $\beta$, abbreviated
$\Gamma_{d}(\alpha,\beta)$-distribution.

If $\beta$ is constant, we call $\mu$ a \emph{$\|\cdot\|$-homogeneous
$\Gamma_{d}(\alpha,\beta)$-distribution}.
\end{definition}

Observe that the notation $\Gamma_{d}(\alpha,\beta)$ implicitly also specifies
which norm we use, because $\alpha$ is a measure on the unit sphere with
respect to the norm employed and $\beta$ is a function on it. The parameters
$\alpha$ and $\beta$ play a comparable role as shape and scale parameters as
in the usual positive univariate case.

\begin{remark}
(i) Obviously the Lévy measure $\nu_{\mu}$ of $\mu$ is given by
\begin{equation}
\nu_{\mu}(E)=\int_{\mathbf{S}_{\mathbb{R}^{d},\Vert\cdot\Vert}}\int
_{\mathbb{R}_{+}}1_{E}(rv)\frac{e^{-\beta(v)r}}{r}\mathrm{d}r\alpha
(\mathrm{d}v) \label{eq:gamlevmeas}%
\end{equation}
for all $E\in\mathcal{B}(\mathbb{R}^{d})$. This expression is equivalent to
\begin{equation}
\nu_{\mu}(\mathrm{d}x)=\frac{e^{-\beta(x/\left\Vert x\right\Vert )\left\Vert
x\right\Vert }}{\left\Vert x\right\Vert }\widetilde{\alpha}(\mathrm{d}x),\quad
x\in\mathbb{R}^{d} \label{MEquivLevMeasMeas}%
\end{equation}
where $\widetilde{\alpha}$ is a measure on $\mathbb{R}^{d}$ given by
\begin{equation}
\widetilde{\alpha}(E)=\int_{\mathbf{S}_{\mathbb{R}^{d},\left\Vert
\cdot\right\Vert }}\int_{0}^{\infty}1_{E}(rv)\mathrm{d}r\alpha(\mathrm{d}%
v),\quad E\in\mathcal{B}(\mathbb{R}^{d}). \label{MLVGK2}%
\end{equation}

(ii) Likewise we define $\mathbb{M}_{d}(\mathbb{R})$ and $\mathbb{S}_{d}%
$-valued Gamma distributions with parameters $\alpha$ and $\beta$ (abbreviated
$\Gamma_{\mathbb{M}_{d}}(\alpha,\beta)$ and $\Gamma_{\mathbb{S}_{d}}%
(\alpha,\beta)$, respectively) by replacing $\mathbb{R}^{d}$ with
$\mathbb{M}_{d}(\mathbb{R})$ and $\mathbb{S}_{d}$, respectively, and the
Euclidean scalar product with $\left\langle Z,X\right\rangle =$%
\textrm{$\mathrm{tr}$}$(X^{\top}Z)$. All upcoming results immediately
generalise to this matrix-variate setting. We provide further details in
Section \ref{sec:matrix}.
\end{remark}

If $d=1$ and $\alpha(\{-1\})=0$, then we have the usual one-dimensional
$\Gamma(\alpha(\{1\}),\beta(1))$-distribution. In general it is elementary to
see that for $d=1$ a random variable $X\sim\Gamma_{1}(\alpha,\beta)$ if and
only if $X\overset{\mathcal{D}}{=}X_{1}-X_{2}$ with $X_{1}\sim\Gamma
(\alpha(\{1\}),\beta(1))$ and $X_{2}\sim\Gamma(\alpha(\{-1\},\beta(-1))$ being
two independent usual Gamma random variables, i.e. $X$ has a bilateral Gamma
distribution as analysed in \cite{KuechlerTappe2007,KuechlerTappe2008} and
introduced in \cite{CGMY2002,MadanCarrChang1998} under the name variance Gamma
distribution. If $\alpha(\{1\})=\alpha(\{-1\})$ and $\beta(1)=\beta(-1)$, it
indeed can be represented as the variance mixture of a normal random variable
with an independent positive Gamma one (a comprehensive summary of this case
can be found in \cite{SteutelVanHarn2004} where it is called sym-Gamma distribution).

Now we address the question of which $\alpha,\beta$ we can take to obtain a
Gamma distribution.

\begin{proposition}
\label{ExitMulGam} Let $\alpha$ be a finite measure on ${\mathbf{S}%
}_{\mathbb{R}^{d},\Vert\cdot\Vert}$ and $\beta:{\mathbf{S}}_{\mathbb{R}%
^{d},\Vert\cdot\Vert}\rightarrow\mathbb{R}_{+}$ a measurable function. Then
\eqref{eq:gamlevmeas} defines a Lévy measure $\nu_{\mu}$ and thus there exists
a $\Gamma_{d}(\alpha,\beta)$ probability distribution $\mu$ if and only if
\begin{equation}
\int_{{\mathbf{S}}_{\mathbb{R}^{d},\Vert\cdot\Vert}}\ln\left(  1+\frac
{1}{\beta(v)}\right)  \alpha(\mathrm{d}v)<\infty. \label{eq:condbeta}%
\end{equation}
Moreover, $\int_{\mathbb{R}^{d}}(\Vert x\Vert\wedge1)\nu_{\mu}(\mathrm{d}%
x)<\infty$ holds true.
\end{proposition}

The condition \eqref{eq:condbeta} is trivially satisfied, if $\beta$ is
bounded away from zero $\alpha$-almost everywhere.

\begin{proof}%
\begin{align*}
\int_{\Vert x\Vert\leq1}\Vert x\Vert\nu_{\mu}(\mathrm{d}x)  &  =\int
_{\mathbf{S}_{\mathbb{R}^{d},\Vert\cdot\Vert}}\int_{0}^{1}e^{-\beta
(v)r}\mathrm{d}r\alpha(\mathrm{d}v)=\int_{\mathbf{S}_{\mathbb{R}^{d}%
,\Vert\cdot\Vert}}\frac{1-e^{-\beta(v)}}{\beta(v)}\alpha(\mathrm{d}v)\\
&  \leq\alpha(\mathbf{S}_{\mathbb{R}^{d},\Vert\cdot\Vert})<\infty
\end{align*}
using the elementary inequality $1-e^{-x}\leq x,\,$for each $x\in
\mathbb{R}_{+}$. Denoting by $E_{1}$ the exponential integral function given
by $E_{1}(z)=\int_{z}^{\infty}\frac{e^{-t}}{t}dt$ for $z\in\mathbb{R}_{+}$, we
get
\begin{align}
\int_{\Vert x\Vert>1}\nu_{\mu}(\mathrm{d}x)  &  =\int_{\mathbf{S}%
_{\mathbb{R}^{d},\Vert\cdot\Vert}}\int_{1}^{\infty}\frac{e^{-\beta(v)r}}%
{r}\mathrm{d}r\alpha(\mathrm{d}v)=\int_{\mathbf{S}_{\mathbb{R}^{d},\Vert
\cdot\Vert}}E_{1}(\beta(v))\alpha(\mathrm{d}v)\\
&  =\int_{0}^{\infty}E_{1}(z)\tau(\mathrm{d}z),
\end{align}
where we made the substitution $z=\beta(v)$ and $\tau(E)=\alpha(\beta
^{-1}(E))$ for all Borel sets $E$ in $\mathbb{R}_{+}$. Since $\tau$ is a
finite measure and $0\leq E_{1}(z)\leq e^{-z}\ln(1+1/z)\,\forall
\,z\in\mathbb{R}_{+}$ (see \cite[p. 229]{AbramowitzStegun}),
\[
\int_{1/2}^{\infty}E_{1}(z)\tau(\mathrm{d}z)<\infty.
\]
The series representation $E_{1}(z)=-\gamma-\ln(z)-\sum_{i=1}^{\infty}%
\frac{(-1)^{i}z^{i}}{n\cdot n!}$ with $\gamma$ being the Euler-Mascheroni
constant (\cite[p. 229]{AbramowitzStegun}) implies that $\lim_{z\downarrow
0}E_{1}(z)/(-\ln(z))=1$. Hence,
\[
\int_{0}^{1/2}E_{1}(z)\tau(\mathrm{d}z)<\infty\Leftrightarrow\int_{0}%
^{1/2}|\ln(z)|\tau(\mathrm{d}z)<\infty\Leftrightarrow\int_{0}^{1/2}%
\ln(1+1/z)\tau(\mathrm{d}z)<\infty
\]
using $\ln(1+1/z)=\ln(1+z)-\ln(z)$ and the finiteness of $\tau$ in the second
equivalence. Appealing to the finiteness of $\tau$ once more, the above
conditions are equivalent to
\[
\int_{0}^{\infty}\ln(1+1/z)\tau(\mathrm{d}z)=\int_{\mathbf{S}_{\mathbb{R}%
^{d},\Vert\cdot\Vert}}\ln(1+1/\beta(v))\alpha(\mathrm{d}v)<\infty.
\]

\end{proof}

The next proposition shows that the definition of a Gamma distribution does
not depend on the norm, only the parametrisation changes when using different norms.

\begin{proposition}
\label{th:normchange} Let $\Vert\cdot\Vert_{a}$ be a norm on $\mathbb{R}^{d}$
and $\mu$ be a $\Gamma_{d}(\alpha,\beta)$ distribution with $\alpha$ being a
finite measure on ${\mathbf{S}}_{\mathbb{R}^{d},\Vert\cdot\Vert_{a}}$ and
$\beta:{\mathbf{S}}_{\mathbb{R}^{d},\Vert\cdot\Vert_{a}}\rightarrow
\mathbb{R}_{+}$ measurable. If $\Vert\cdot\Vert_{b}$ is another norm on
$\mathbb{R}^{d}$, then $\mu$ is a $\Gamma_{d}(\alpha_{b},\beta_{b})$
distribution with $\alpha_{b}$ being a finite measure on ${\mathbf{S}%
}_{\mathbb{R}^{d},\Vert\cdot\Vert_{b}}$ and $\beta_{b}:{\mathbf{S}%
}_{\mathbb{R}^{d},\Vert\cdot\Vert_{b}}\rightarrow\mathbb{R}_{+}$ measurable.
Moreover, it holds that
\begin{align}
\alpha_{b}(E)  &  =\int_{{\mathbf{S}}_{\mathbb{R}^{d},\Vert\cdot\Vert_{a}}%
}1_{E}\left(  \frac{v}{\Vert v\Vert_{b}}\right)  \alpha(\mathrm{d}v)\, &
\forall\,E  &  \in\mathcal{B}({\mathbf{S}}_{\mathbb{R}^{d},\Vert\cdot\Vert
_{b}})\label{eq:anc}\\
\beta_{b}(v_{b})  &  =\beta\left(  \frac{v_{b}}{\Vert v_{b}\Vert_{a}}\right)
\Vert v_{b}\Vert_{a}\, & \forall\,v_{b}  &  \in{\mathbf{S}}_{\mathbb{R}%
^{d},\Vert\cdot\Vert_{b}}. \label{eq:bnc}%
\end{align}

\end{proposition}

The above formulae show that the mass in the different directions, which is
given by $\alpha$, does not change, and $\beta$ only needs to be adapted for
the scale changes implied by the change of the norm.

\begin{proof}
Substituting first $v_{b}=v/\Vert v\Vert_{b}$ and then $s=r/\Vert v_{b}%
\Vert_{a}$ gives:
\begin{align*}
&  \exp\left(  \int_{\mathbf{S}_{\mathbb{R}^{d},\Vert\cdot\Vert_{a}}}%
\int_{\mathbb{R}_{+}}\left(  e^{irv^{\top}z}-1\right)  \frac{e^{-\beta(v)r}%
}{r}\mathrm{d}r\alpha(\mathrm{d}v)\right) \\
&  \qquad\qquad=\exp\left(  \int_{\mathbf{S}_{\mathbb{R}^{d},\Vert\cdot
\Vert_{b}}}\int_{\mathbb{R}_{+}}\left(  e^{i\frac{r}{\Vert v_{b}\Vert_{a}%
}v_{b}^{\top}z}-1\right)  \frac{e^{-\beta\left(  \frac{v_{b}}{\Vert v_{b}%
\Vert_{a}}\right)  r}}{r}\mathrm{d}r\alpha_{b}(\mathrm{d}v_{b})\right) \\
&  \qquad\qquad=\exp\left(  \int_{\mathbf{S}_{\mathbb{R}^{d},\Vert\cdot
\Vert_{b}}}\int_{\mathbb{R}_{+}}\left(  e^{isv_{b}^{\top}z}-1\right)
\frac{e^{-\beta\left(  \frac{v_{b}}{\Vert v_{b}\Vert_{a}}\right)  \Vert
v_{b}\Vert_{a}s}}{s}\mathrm{d}s\alpha_{b}(\mathrm{d}v_{b})\right)  .
\end{align*}

\end{proof}

\subsection{Properties}

In this section we study several fundamental properties of our Gamma distributions.

\begin{proposition}
\label{SelfDecpMuG} Any $\Gamma_{d}(\alpha,\beta)$-distribution is self-decomposable.
\end{proposition}

\begin{proof}
This follows immediately from the definition and \cite[Th. 15.10]{sato1999}.
\end{proof}

Later on we will considerably improve this result by showing that we are in a
very special subset of the self-decomposable distributions. This result has
important implications for applications where one likes to work with
distributions having densities, i.e. distributions which are absolutely
continuous (with respect to the Lebesgue measure).

\begin{proposition}
\label{DenMuG} Assume that $\operatorname{supp}{\alpha}$ is of full dimension,
i.e. that it contains $d$ linearly independent vectors in $\mathbb{R}^{d}$.
Then the $\Gamma_{d}(\alpha,\beta)$-distribution is absolutely continuous.
\end{proposition}

\begin{proof}
It is immediate that the support of $\ \Gamma_{d}(\alpha,\beta)$ is the closed
convex cone generated by $\operatorname{supp}{\alpha}$. Hence, the support of
$\Gamma_{d}(\alpha,\beta)$ is of full dimension and so the distribution is
non-degenerate. Thus \cite{Sato1982} concludes.
\end{proof}

It follows along the same lines that in the degenerate case the $\Gamma
_{d}(\alpha,\beta)$-distribution is absolutely continuous with respect to the
Lebesgue measure on the subspace generated by $\operatorname{supp}{\alpha}$.
If $\operatorname{supp}{\alpha}$ consists of exactly $d$ linearly independent
vectors, $\Gamma_{d}(\alpha,\beta)$ equals the distribution of a linear
transformation of a vector of $d$ independent univariate Gamma random
variables with appropriate parameters and thus the density can be calculated
easily using the density transformation theorem with an invertible linear map.
If $\operatorname{supp}{\alpha}$ is a finite set of full dimension, one can
calculate the density from the density of independent univariate Gamma random
variables by using the density transformation theorem with an invertible
linear map and integrating out the non-relevant dimensions. In general the
density can be determined via solving a partial integro-differential equation
(see \cite{SatoYamazato1984}). Moreover, criteria for qualitative properties
of the density like continuity and continuous differentiability can be deduced
from the results of \cite{Sato1979,Sato1980}, but looking at the simple case
of a vector of independent univariate Gamma distributions one immediately sees
that the sufficient conditions given there are far from being sharp. Therefore
we refrain from giving more details.

Next we show that our $d$-dimensional Gamma distribution has the same
closedness properties regarding scaling and convolution as the usual
univariate one.

\begin{proposition}
(i) Let $X\sim\Gamma_{d}(\alpha,\beta)$ and $c>0$. Then $cX\sim\Gamma
_{d}(\alpha,\beta/c)$.

(ii) Let $X_{1}\sim\Gamma_{d}(\alpha_{1},\beta)$ and $X_{2}\sim\Gamma
_{d}(\alpha_{2},\beta)$ be two independent $d$-dimensional Gamma variables.
Then $X_{1}+X_{2}\sim\Gamma_{d}(\alpha_{1}+\alpha_{2},\beta)$.
\end{proposition}

\begin{proof}
Follows immediately from considering the characteristic functions.
\end{proof}

Likewise it is immediate to see the following distributional properties of the
induced L\'evy process.

\begin{proposition}
Let $L$ be a $\Gamma_{d}(\alpha,\beta)$ L\'evy process, i.e. $L_{1}\sim
\Gamma_{d}(\alpha,\beta)$. Then $L_{t}\sim\Gamma_{d}(t\alpha,\beta)$ for all
$t\in\mathbb{R}_{+}$.
\end{proposition}

Of high importance for applications is that the class of $\Gamma_{d}$
distributions is invariant under invertible linear transformations.

\begin{proposition}
\label{LinTransGam} Let $X\sim\Gamma_{d}(\alpha,\beta)$ (with respect to the
norm $\Vert\cdot\Vert$) and $A$ be an invertible $d\times d$ matrix. Then
$AX\sim\Gamma_{d}(\alpha_{A},\beta_{A})$ with respect to the norm $\Vert
\cdot\Vert_{A}=\Vert A^{-1}\cdot\Vert$ and
\begin{align}
\alpha_{A}(E)  &  =\int_{{\mathbf{S}}_{\mathbb{R}^{d},\Vert\cdot\Vert}}%
1_{E}\left(  Av\right)  \alpha(\mathrm{d}v)\,=\alpha(A^{-1}E) & \forall\,E  &
\in\mathcal{B}({\mathbf{S}}_{\mathbb{R}^{d},\Vert\cdot\Vert_{A}})\label{3.5}\\
\beta_{A}(v)  &  =\beta\left(  A^{-1}v\right)  & \forall\,v  &  \in
{\mathbf{S}}_{\mathbb{R}^{d},\Vert\cdot\Vert_{A}}. \label{3.6}%
\end{align}

\end{proposition}

\begin{proof}
We have for all $z\in\mathbb{R}^{d}$
\begin{align*}
E\left(  e^{i<z,AX>}\right)  =  &  \int_{\mathbb{R}^{d}}e^{i<z,Ax>}%
\mu(dx)=\int_{\mathbf{S}_{\mathbb{R}^{d},\Vert\cdot\Vert}}\int_{\mathbb{R}%
_{+}}\left(  e^{irv^{\top}A^{\top}z}-1\right)  \frac{e^{-\beta(v)r}}%
{r}dr\alpha(dv)\\
=  &  \int_{\mathbf{S}_{\mathbb{R}^{d},\Vert\cdot\Vert_{A}}}\int
_{\mathbb{R}_{+}}\left(  e^{iru^{\top}z}-1\right)  \frac{e^{-\beta(A^{-1}u)r}%
}{r}dr\alpha(A^{-1}du)\\
=  &  \int_{\mathbf{S}_{\mathbb{R}^{d},\Vert\cdot\Vert_{A}}}\int
_{\mathbb{R}_{+}}\left(  e^{iru^{\top}z}-1\right)  \frac{e^{-\beta_{A}(u)r}%
}{r}dr\alpha_{A}(du)
\end{align*}
where we substituted $u=Av$.
\end{proof}

It is easy to see that the above proposition can be extended to $m\times d$
matrices of full rank with $m>d$. Obviously, such a result cannot hold in
general for a linear transformation $A$ with $\ker(A)\not =\{0\}$, since
combinations of one dimensional Gamma distributions are in general not
univariate Gamma distributions.

Next we present an alternative representation of the characteristic function.

\begin{proposition}
\label{th:fouriertransf} Let $\mu$ be $\Gamma_{d}(\alpha,\beta)$ distributed.
Then the characteristic function is given by
\begin{equation}
\hat{\mu}(z)=\exp\left(  \int_{{\mathbf{S}}_{\mathbb{R}^{d},\Vert\cdot\Vert}%
}\ln\left(  \frac{\beta(v)}{\beta(v)-iv^{\top}z}\right)  \alpha(\mathrm{d}%
v)\right)  \,\mbox{ for all }\,z\in\mathbb{R}^{d} \label{eq:fourier}%
\end{equation}
where $\ln$ is the main branch of the complex logarithm.
\end{proposition}

\begin{proof}
Follows from the definition and\ the well known fact%
\[
\int_{0}^{\infty}\left(  e^{-r(-iv^{\top})z}-1\right)  \frac{e^{-\beta(v)r}%
}{r}\mathrm{d}r=\ln\left(  \frac{\beta(v)}{\beta(v)-iv^{\top}z}\right)  .
\]

\end{proof}

Note that if $\alpha$ has countable support $\{v_{j}\}_{j\in\mathbb{N}}$,
then
\[
\hat{\mu}(z)=\prod_{j\in\mathbb{N}}\left(  \frac{\beta(v_{j})}{\beta
(v_{j})-iv_{j}^{\top}z}\right)  ^{\alpha(\{v_{j}\})}.
\]

We now show that the Fourier-Laplace transform of a Gamma distribution exists
if and to a certain extent only if $\beta$ is bounded away from zero $\alpha$
almost everywhere.

\begin{theorem}
\label{th:exfourlapl} (i) The Fourier-Laplace transform $\hat{\mu}$ of a
$\Gamma_{d}(\alpha,\beta)$ distribution $\mu$ exists for all $z$ in a
neighborhood $U\subseteq\mathbb{C}^{d}$ of zero, if $\beta(v)\geq\kappa$ for
$v\in{\mathbf{S}}_{\mathbb{R}^{d},\Vert\cdot\Vert}$ $\alpha$-a.e. with
$\kappa>0$. $\hat{\mu}$ is analytic there and given by formula \eqref{eq:fourier}.

(ii) If there exists a sequence $(v_{n})_{n\in\mathbb{N}}$ in ${\mathbf{S}%
}_{\mathbb{R}^{d},\Vert\cdot\Vert}$ with $\lim_{n\rightarrow\infty}\beta
(v_{n})=0$ and $\alpha(\{v_{n}\})>0$ for all $n\in\mathbb{N}$, then the
Fourier-Laplace transform $\hat{\mu}$ exists in no neighborhood $U\subseteq
\mathbb{C}^{d}$ of zero.
\end{theorem}

\begin{proof}
Using Proposition \ref{th:normchange} we can assume w.l.o.g. that the
Euclidean norm $\Vert\cdot\Vert_{2}$ is used for the definition of the
$\Gamma_{d}(\alpha,\beta)$ distribution.

(i) We will now show (i) for $U=B_{\kappa}(0)\subseteq\mathbb{C}^{d}$, where
$B_{\kappa}(0):=\{x\in\mathbb{C}^{d}:\Vert x\Vert_{2}<\kappa\}$. From
Proposition \ref{th:fouriertransf} it is clear that $\hat{\mu}(z)$ exists for
all $z\in B_{\kappa}(0)\subseteq\mathbb{C}^{d}$, if and only if
\[
\int_{\mathbf{S}_{\mathbb{R}^{d},\Vert\cdot|_{2}}}\ln\left(  \frac{\beta
(v)}{\beta(v)-iv^{\top}z}\right)  \alpha(\mathrm{d}v)=-\int_{\mathbf{S}%
_{\mathbb{R}^{d},\Vert\cdot\Vert_{2}}}\ln\left(  1-\frac{iv^{\top}z}{\beta
(v)}\right)  \alpha(\mathrm{d}v)
\]
exists for all $z\in B_{\kappa}(0)$. Consider now an arbitrary $\delta
\in(0,1)$ and $z\in\overline{B}_{\delta\kappa}(0)$. Then the Cauchy-Schwarz
inequality implies $|iv^{\top}z|\leq\Vert z\Vert_{2}\leq\delta\kappa$ and
hence $|(iv^{\top}z)/\beta(v)|\leq\delta$. Therefore $\ln\left(
1-\frac{iv^{\top}z}{\beta(v)}\right)  $ exists and is bounded on $\overline
{B}_{\delta\kappa}(0)$ $\alpha$-a.e. This implies that
\[
-\int_{\mathbf{S}_{\mathbb{R}^{d},\Vert\cdot\Vert_{2}}}\ln\left(
1-\frac{iv^{\top}z}{\beta(v)}\right)  \alpha(\mathrm{d}v)
\]
exists on $\overline{B}_{\delta\kappa}(0)$. Since $\delta\in(0,1)$ was
arbitrary, this concludes the proof of (i), since the analyticity follows
immediately from the appendix of \cite{dfs}.

(ii) W.l.o.g. assume $\beta(v_{n})<1/n$. For $n\in\mathbb{N}$ set
$z_{n}=-i\beta(v_{n})v_{n}$. Then $\Vert z_{n}\Vert_{2}=\beta(v_{n})<1/n$ and
$1-(iv_{n}^{\top}z_{n})/\beta(v_{n})=0$. Hence,
\[
\int_{\{v_{n}\}}\ln\left(  1-\frac{iv^{\top}z_{n}}{\beta(v)}\right)
\alpha(\mathrm{d}v)
\]
and thereby
\[
\int_{{\mathbf{S}_{\mathbb{R}^{d},\Vert\cdot\Vert_{2}}}}\ln\left(
1-\frac{iv^{\top}z_{n}}{\beta(v)}\right)  \alpha(\mathrm{d}v)
\]
do not exist. This implies that $\hat{\mu}$ is not defined on $B_{1/n}(0)$.
Since $n\in\mathbb{N}$ was arbitrary, this shows (ii).
\end{proof}

\begin{remark}
To extend this result to the matrix case, one simply has to use the Frobenius
or trace norms and the scalar product $Z,X\mapsto\mathrm{tr}(X^{\top}Z)$
instead of the Euclidean norm and scalar product. We consider this in\ Section
\ref{sec:matrix}.
\end{remark}


\begin{proposition}
\label{MomMulGam}A $\Gamma_{d}(\alpha,\beta)$ distribution $\mu$ has a finite
moment of order $k>0$, i.e. $\int_{\mathbb{R}^{d}}\Vert x\Vert^{k}%
\mu(\mathrm{d}x)<\infty$, if and only if
\begin{equation}
\int_{{\mathbf{S}}_{\mathbb{R}^{d},\Vert\cdot\Vert}}\beta(v)^{-k}%
\alpha(\mathrm{d}v)<\infty. \label{eq:condrmom}%
\end{equation}
Moreover, if $m$ is the mean vector and $\Sigma=(\sigma_{ij})_{i,j=1,\ldots
,d}$ is the covariance matrix of $\ \Gamma_{d}(\alpha,\beta)$
\begin{equation}
m=\int_{\mathbf{S}_{\mathbb{R}^{d},\Vert\cdot\Vert}}\beta(v)^{-1}%
v\alpha(\mathrm{d}v). \label{meanGCC}%
\end{equation}
and%
\begin{equation}
\Sigma=\int_{\mathbf{S}_{\mathbb{R}^{d},\Vert\cdot\Vert}}\beta(v)^{-2}%
vv^{\top}\alpha(\mathrm{d}v) \label{CovGGC}%
\end{equation}

\end{proposition}

\begin{proof}
If $\beta$ is bounded away from zero, \eqref{eq:condrmom} holds trivially and
Theorem \ref{th:exfourlapl} implies that $\mu$ has finite moments of all
orders $k>0$. So w.l.o.g. assume that $\beta$ is not bounded away from zero in
the following. By \cite[p. 162]{sato1999} $\mu$ has a finite moment of order
$k$, if and only if
\[
\int_{\mathbf{S}_{\mathbb{R}^{d},\Vert\cdot\Vert}}\int_{1}^{\infty}r^{k}%
\frac{e^{-\beta(v)r}}{r}\mathrm{d}r\alpha(\mathrm{d}v)<\infty.
\]
Substituting $s=r\beta(v)$ this is equivalent to
\begin{equation}
\int_{\mathbf{S}_{\mathbb{R}^{d},\Vert\cdot\Vert}}\beta(v)^{-k}\int_{\beta
(v)}^{\infty}s^{k-1}e^{-s}\mathrm{d}s\alpha(\mathrm{d}v)<\infty.
\label{eq:condmomh}%
\end{equation}
Assuming without loss of generality that $\beta(v)\leq1$ for all
$v\in\mathbf{S}_{\mathbb{R}^{d},\Vert\cdot\Vert}$, we have that
\[
0<C(k):=\int_{1}^{\infty}s^{k-1}e^{-s}\mathrm{d}s\leq\int_{\beta(v)}^{\infty
}s^{k-1}e^{-s}\mathrm{d}s\leq\Gamma(k).
\]
Hence, \eqref{eq:condmomh} is equivalent to \eqref{eq:condrmom}. Finally,
(\ref{meanGCC}) and (\ref{CovGGC}) follow from Example 25.12 in
\cite{sato1999} and observing that that the infinitely divisible distribution
$\Gamma_{d}(\alpha,\beta)$ with Fourier transform (\ref{FTGMult}) has Lévy
triplet $(\zeta,0,\nu_{\mu}),$ where $\zeta=\int_{\left\Vert x\right\Vert
\leq1}x\nu_{\mu}(\mathrm{d}x).$
\end{proof}

\begin{corollary}
A $\|\cdot\|$-homogeneous $\Gamma_{d}(\alpha,\beta)$ distribution has an
analytic Fourier-La\-place transform in $B_{\beta}(0)$ and finite moments of all orders.
\end{corollary}

Hence, any homogeneous Gamma distribution behaves like one would expect it
from the univariate case. However, the behaviour in the non-homogeneous case
may be drastically different, as the following examples illustrate.

\begin{example}
\label{ex:1} Consider $d=2$. Let $\alpha$ be concentrated on $\{v_{n}%
\}_{n\in\mathbb{N}}$ with
\[
v_{n}=(\sin(n^{-1}),\cos(n^{-1}))
\]
and set $\alpha(\{v_{n}\})=e^{-n}$ and $\beta(v_{n})=1/n$ for all
$n\in\mathbb{N}$. Then by Theorem \ref{th:exfourlapl} (ii) the Fourier-Laplace
transform exists in no neighbourhood of zero.

$\int_{{\mathbf{S}}_{\mathbb{R}^{d},\Vert\cdot\Vert_{2}}}\beta(v)^{-k}%
\alpha(\mathrm{d}v)=\sum_{n\in\mathbb{N}}n^{k}e^{-n}$ is finite for all $k>0$
using the quotient criterion, because
\[
\lim_{n\rightarrow\infty}\frac{(n+1)^{k}e^{-(n+1)}}{n^{k}e^{-n}}=e^{-1}<1.
\]

Thus, we have moments of all orders, but the Fourier-Laplace transform exists
in no complex neighbourhood of zero.
\end{example}

\begin{example}
Consider the set-up of Example \ref{ex:1}, but set now $\alpha(\{v_{n}%
\})=1/n^{1+m}$ for some real $m>0$. $\int_{{\mathbf{S}}_{\mathbb{R}^{d}%
,\Vert\cdot\Vert_{2}}}\beta(v)^{-k} \alpha(\mathrm{d}v)=\sum_{n\in\mathbb{N}%
}\frac{n^{k}}{n^{1+m}}$ is finite if and only if $k<m$.

It is easy to see that condition \eqref{eq:condbeta} is satisfied if condition
\eqref{eq:condrmom} holds for some $k>0$. Hence, the $\Gamma_{2}(\alpha
,\beta)$ distribution exists indeed, but only moments of orders smaller than
$m$ are finite.
\end{example}

\begin{example}
Consider again the set-up of Example \ref{ex:1}. Set now $\alpha
(\{v_{n}\})=(\ln(1+n)^{3}(n+1))^{-1}$.

Then $\int_{{\mathbf{S}}_{\mathbb{R}^{d},\Vert\cdot\Vert_{2}}}\ln\left(
1+\frac{1}{\beta(v)}\right)  \alpha(\mathrm{d}v)=\sum_{n\in\mathbb{N}}\frac
{1}{\ln(1+n)^{2}(1+n)}<\infty$ (see \cite[Theorem 3.29]{Rudin1976} and thus
the $\Gamma_{2}(\alpha,\beta)$ distribution is well-defined.

Yet, $\int_{{\mathbf{S}}_{\mathbb{R}^{d},\Vert\cdot\Vert_{2}}}\beta
(v)^{-k}\alpha(\mathrm{d}v)=\sum_{n\in\mathbb{N}}\frac{n^{k}}{\ln
(1+n)^{3}(1+n)}=\infty$ for all real $k>0$ and so the $\Gamma_{2}(\alpha
,\beta)$ distribution has no finite moments of positive orders at all.
\end{example}

\section{Gamma and Generalised Gamma Convolutions on Cones}

\label{sec:ggc}

\subsection{Cone-valued infinitely divisible random elements}

We first review several facts about infinitely divisible elements with values
in a cone of a finite dimensional Euclidean space $B$ with norm $\left\Vert
\cdot\right\Vert $ and inner product $\left\langle \cdot,\cdot\right\rangle .$
A nonempty convex set $K$ of $B$ is said to be a \textit{cone} if $\lambda
\geq0$ and $x\in K$ imply $\lambda x\in K.$ A cone is \textit{proper }if $x=0$
whenever $x\,$and $-x\,$are in $K$. The \textit{dual cone} $K^{\prime}$ of $K$
is defined as $K^{\prime}=\left\{  y\in B^{\prime}:\left\langle
y,s\right\rangle \geq0\text{ for every }s\in K\right\}  .$ A proper cone $K$
induces a partial order on $B$ by defining $x_{1}\leq_{K}x_{2}$ whenever
$x_{2}-x_{1}\in K$ for $x_{1}\in B$ and $x_{2}\in B$. \ Examples of proper
cones are $\mathbb{R}_{+}$, $\mathbb{R}_{+}^{d}=[0,\infty)^{d}$,
$\mathbb{S}_{d}^{+}$ and $\overline{\mathbb{S}}_{d}^{+}$. $\ $

A random element $X$ in $K$ is \emph{infinitely divisible (ID)} if and only if
for each integer $p\geq1$ there exist $p$ independent identically distributed
random elements $X_{1},...,X_{p}$ in $K$ such that $X\overset{law}{=}%
X_{1}+...+X_{p}.$ A probability measure $\mu$ on $K$ is ID if it is the
distribution of an ID element in $K.$ It is known (see \cite{S91}) that such a
distribution $\mu$ is concentrated on a cone $K$ if and only if its Laplace
transform $\emph{L}_{\mu}(\Theta)=\int_{K}\exp(-\left<  \Theta,x\right>
)\mu(\mathrm{d}x)$ is given by the \emph{regular Lévy-Khintchine
representation}
\begin{equation}
\emph{L}_{\mu}(\Theta)=\exp\left\{  -\left\langle \Theta,\Psi_{0}\right\rangle
-\int_{K}\left(  1-e^{-\left\langle \Theta,x\right\rangle }\right)  \nu_{\mu
}(\mathrm{d}x)\right\}  \text{\quad for all }\Theta\in K^{\prime},
\label{gllkr}%
\end{equation}
where $\Psi_{0}\in K$ and the Lévy measure is such that $\nu_{\mu}(K^{c})=0$
and
\begin{equation}
\int_{K}(\left\Vert x\right\Vert \wedge1)\,\nu_{\mu}(\mathrm{d}x)<\infty.
\label{sL1}%
\end{equation}

If $X=\left\{  X(t);t\geq0\right\}  $ is the $K$-increasing Lévy process
($K$-valued subordinator) associated to $\mu$, its Lévy-Itô decomposition is
of the form
\begin{align}
X(t)  &  =t\Psi_{0}+\int_{0}^{t}\int_{K}xN(\mathrm{d}t,\mathrm{d}x)\nonumber\\
&  =t\Psi_{0}+\sum_{s\leq t}\Delta X(s)\quad a.s., \label{itorep}%
\end{align}
where $\Delta X(s)\in K$ for all $s\geq0$ a.s. and $N(dt,\mathrm{d}x)$ is a
Poisson random measure on $\mathbb{R}_{+}\times K$\ with%
\begin{equation}
\mathbb{E}\{N(\mathrm{d}t,\mathrm{d}x)\}=\nu_{\mu}(\mathrm{d}x)\mathrm{d}t.
\label{ExpPRM}%
\end{equation}

\subsection{Cone-valued Gamma distributions}

Let $\mathbf{S}_{\left\Vert \cdot\right\Vert }$ be the unit sphere of $B$ with
respect to the norm $\left\Vert \cdot\right\Vert $ and let $K$ be a proper
cone of $B.$ We write $\mathbf{S}_{\left\Vert \cdot\right\Vert }%
^{K}=\mathbf{S}_{\left\Vert \cdot\right\Vert }\cap K$ and denote by
$\mathcal{B}(\mathbf{S}_{\left\Vert \cdot\right\Vert }^{K})$ the Borel sets of
$\mathbf{S}_{\left\Vert \cdot\right\Vert }^{K}$,

Similar to Definition \ref{def:mgamma} we have Gamma distributions in the cone
$K.$

\begin{definition}
\label{def:gammaconvalued} Let $\mu$ be an infinitely divisible distribution
on the cone $K$. If there exist a finite measure $\alpha$ on $\mathbf{S}%
_{\left\Vert \cdot\right\Vert }^{K},$ and a measurable function $\beta
:\mathbf{S}^{K}_{\left\Vert \cdot\right\Vert }\rightarrow\mathbb{R}_{+}$ such
that
\begin{equation}
\emph{L}_{\mu}(\Theta)=\exp\left\{  -\int_{\mathbf{S}_{\left\Vert
\cdot\right\Vert }^{K}}\int_{0}^{\infty}\left(  1-e^{-r\left\langle
\Theta,U\right\rangle }\right)  \frac{e^{-\beta(U)r}}{r}\mathrm{d}%
r\alpha(\mathrm{d}U)\right\}  \label{LTMX}%
\end{equation}
for all $\Theta\in K^{\prime},$ then $\mu$ is called a $K$-Gamma distribution
with parameters $\alpha$ and $\beta$, and we write $\mu\sim\Gamma_{K}%
(\alpha,\beta)$. $\ $The Lévy measure $\nu_{\mu}$ of $\mu$ is
\begin{equation}
\nu_{\mu}(E)=\int_{\mathbf{S}_{\left\Vert \cdot\right\Vert }^{K}}\int
_{0}^{\infty}1_{E}(rU)\frac{e^{-\beta(U)r}}{r}\mathrm{d}r\alpha(\mathrm{d}%
U),\quad E\in\mathcal{B}(K), \label{LevMeaMatPol}%
\end{equation}
and satisfies
\begin{equation}
\int_{K}\min(1,\left\Vert x\right\Vert )\nu_{\mu}(\mathrm{d}x)<\infty.
\label{suL1}%
\end{equation}

\end{definition}

The expression (\ref{LevMeaMatPol}) is equivalent to
\begin{equation}
\nu_{\mu}(\mathrm{d}X)=\frac{e^{-\beta(X/\left\Vert X\right\Vert )\left\Vert
X\right\Vert }}{\left\Vert X\right\Vert }1_{K}(X)\widetilde{\alpha}%
(\mathrm{d}X), \label{EquivLevMeasMeas}%
\end{equation}
where $\widetilde{\alpha}$ is a measure on $K$ given by
\begin{equation}
\widetilde{\alpha}(E)=\int_{\mathbf{S}_{\left\Vert \cdot\right\Vert }^{K}}%
\int_{0}^{\infty}1_{E}(rU)\mathrm{d}r\alpha(\mathrm{d}U),\quad E\in
\mathcal{B}(K). \label{LVGK2}%
\end{equation}

All properties of the multivariate Gamma distribution in Section
\ref{sec:MGammadist} are also true for the cone-valued Gamma distribution. As
in Proposition 3.3 we can in particular show that there exists a $\Gamma
_{K}(\alpha,\beta)$ probability measure $\mu$ if and only if
\begin{equation}
\int_{\mathbf{S}_{\left\Vert \cdot\right\Vert }^{K}}\ln\left(  1+\frac
{1}{\beta(U)}\right)  \alpha(\mathrm{d}U)<\infty, \label{CondExistMG}%
\end{equation}
in which case we have (\ref{suL1}). Also, as for Proposition 3.10 the Laplace
transform of a $\Gamma_{K}(\alpha,\beta)$ probability measure $\mu$ is also
given by
\begin{equation}
\emph{L}_{\mu}(\Theta)=\exp\left\{  -\int_{\mathbf{S}_{\left\Vert
\cdot\right\Vert }^{K}}\ln\left(  1+\frac{\left\langle \Theta,U\right\rangle
)}{\beta(U)}\right)  \alpha(\mathrm{d}U)\right\}  ,\quad\Theta\in K^{\prime}.
\label{AltLTMat}%
\end{equation}

If $\left\Vert \cdot\right\Vert _{b}$ is another norm on $K$ and if $\mu$ has
distribution $\Gamma_{K}(\alpha,\beta)$, then $\mu$ has distribution
$\Gamma_{K}(\alpha_{b},\beta_{b})$ where $\alpha_{b},\beta_{b}$ are given as
in \eqref{eq:anc} and \eqref{eq:bnc} respectively. Also, $\Gamma_{K}%
(\alpha,\beta)$, has a finite moment of order $k>0,$ if and only if
\begin{equation}
\int_{\mathbf{S}_{\left\Vert \cdot\right\Vert }^{K}}\beta(U)^{-k}%
\alpha(\mathrm{d}U)<\infty. \label{MomConMat}%
\end{equation}
In the homogeneous case, i.e. $\beta(U)=\beta_{0}>0$ for any \ $U\in\mathbf{S}%
_{\left\Vert \cdot\right\Vert }^{K},$  we have $\mathbb{E}\left\Vert M\right\Vert
^{k}<\infty$ for any $k>0.$

If $M$ is a random element in $K$ with distribution $\Gamma_{K}(\alpha,\beta)$
and (\ref{MomConMat}) is satisfied with $k=1,$
\begin{equation}
\mathbb{E}(M)=\int_{\mathbf{S}_{\left\Vert \cdot\right\Vert }^{K}}%
\beta(U)^{-1}U\alpha(\mathrm{d}U). \label{FirstMomtGM}%
\end{equation}

\subsection{Itô-Wiener-Gamma integrals}

In this section we formulate an Itô-Wiener-Gamma integral for $K$-valued Gamma
process, similar to the Itô-Wiener-Gamma integral (\ref{WinGamInt}) with
respect to the one-dimensional Gamma process.

Let $\gamma=\gamma(\alpha,\beta)=\left(  \gamma_{t};t\geq0\right)  $ be a
$K$-valued Gamma process. That is, $\gamma$ is the $K$-increasing Lévy process
such that $\mu_{\alpha,\beta}=\Gamma_{K}(\alpha,\beta)$ is the distribution of
$\gamma_{1}.$ Let $N_{\gamma}(\mathrm{d}s,\mathrm{d}x)$ be the random measure
on $\mathbb{R}_{+}\times K$ associated to the $K$-valued jumps of $\gamma$ and
$\nu_{\mu_{\alpha,\beta}}$ be the Lévy measure of $\gamma_{1}$. Hence
$\mathbb{E}\{N(\mathrm{d}t,\mathrm{d}x)\}=\nu_{\mu_{\alpha,\beta}}%
(\mathrm{d}x)\mathrm{d}t$, where
\[
\nu_{\mu_{\alpha,\beta}}(E)=\int_{\mathbf{S}_{\left\Vert \cdot\right\Vert
}^{K}}\int_{0}^{\infty}1_{E}(rU)\frac{e^{-\beta(U)r}}{r}\mathrm{d}%
r\alpha(\mathrm{d}U),\quad E\in\mathcal{B}(K).
\]

Let $h:\mathbb{R}_{+}\times\mathbf{S}_{\left\Vert \cdot\right\Vert }%
^{K}\rightarrow\mathbb{R}_{+}$ be a measurable function such that
\begin{equation}
\int_{{\mathbf{S}}_{\Vert\cdot\Vert}^{K}}\int_{0}^{\infty}\ln\left(
1+\frac{h(s,U)}{\beta(U)}\right)  \alpha(\mathrm{d}U)\mathrm{d}s<\infty,
\label{IntPRM}%
\end{equation}
in which case we say that $h$ belongs to $L(\Gamma_{K}(\alpha,\beta))$. The
last condition is the cone analogon of the one-dimensional condition
(\ref{intcondh}).

We prove in the next proposition that the following Itô-Wiener-Gamma integral
type is well defined
\begin{equation}
Y^{h}=\int_{0}^{\infty}\int_{K}h\left(  s,\frac{x}{\left\Vert x\right\Vert
}\right)  xN(\mathrm{d}s,\mathrm{d}x), \label{wginrep}%
\end{equation}
in the the framework of integration with respect to infinitely divisible
independently scattered random measures (i.d.i.s.r.m.) in Rajput and Rosinski
\cite{RaRo89} (see also \cite{BnSe09b} for the special case of random matrices).

\begin{proposition}
\label{ExitInt1} The integral \eqref{wginrep} is well defined if and only if
the function $h:\mathbb{R}_{+}\times\mathbf{S}_{\left\Vert \cdot\right\Vert
}^{K}\rightarrow\mathbb{R}_{+}$ belongs to $L(\Gamma_{K}(\alpha,\beta))$.

Moreover, $h\in L(\Gamma_{K}(\alpha,\beta))$ if and only if
\begin{equation}
\int_{0}^{\infty}\int_{K}\min(1,\left\Vert h(s,x/\left\Vert x\right\Vert
)x\right\Vert )\nu_{\mu_{\alpha,\beta}}(\mathrm{d}x)\mathrm{d}s<\infty,
\label{CondLebMeah}%
\end{equation}
and (\ref{CondLebMeah}) is equivalent to the following two conditions
\begin{equation}
\int_{{\mathbf{S}}_{\Vert\cdot\Vert}^{K}}\int_{0}^{1/2}\left\vert \ln\left(
z\right)  \right\vert G_{U}(\mathrm{d}z)\alpha(\mathrm{d}U)<\infty
\label{CondLebMeaha}%
\end{equation}
and%
\begin{equation}
\int_{{\mathbf{S}}_{\Vert\cdot\Vert}^{K}}\int_{1/2}^{\infty}\frac{1}{z}%
G_{U}(\mathrm{d}z)\alpha(\mathrm{d}U)<\infty\label{CondLebMeahb}%
\end{equation}
where $G_{U}(\mathrm{d}z)$ is the \textit{measure} on $\mathbb{R}_{+}$ which
is the image of the Lebesgue measure on $\mathbb{R}_{+}$ under the change of
variable $s\rightarrow\beta(U)/h(s,U).$
\end{proposition}

\begin{proof}
By \cite{RaRo89,BnSe09b} the existence of the integral is equivalent to \eqref{CondLebMeah}.

Since $\beta(U)<0$ $a.e$. $U$, Fubini's theorem and elementary computations
give
\begin{align}
I  &  =\int_{0}^{\infty}\int_{K}\min(1,\left\Vert h(s,x/\left\Vert
x\right\Vert )x\right\Vert )\nu_{\mu_{\alpha,\beta}}(\mathrm{d}x)\mathrm{d}s\\
&  =\int_{{\mathbf{S}}_{\Vert\cdot\Vert}^{K}}\int_{0}^{\infty}\int_{0}%
^{\infty}\min(1,\left\Vert h(s,U)rU\right\Vert )\frac{e^{-r\beta(U)}}%
{r}\mathrm{d}r\alpha(\mathrm{d}U)\mathrm{d}s\nonumber\\
&  =\int_{{\mathbf{S}}_{\Vert\cdot\Vert}^{K}}\int_{0}^{\infty}\int
_{0}^{1/h(s,U)}h(s,U)e^{-r\beta(U)}\mathrm{d}r\alpha(\mathrm{d}U)\nonumber\\
&  +\int_{{\mathbf{S}}_{\Vert\cdot\Vert}^{K}}\int_{0}^{\infty}\int
_{1/h(s,U)}^{\infty}\frac{e^{-r\beta(U)}}{r}\mathrm{d}r\alpha(\mathrm{d}%
U)\nonumber\\
&  \overset{z=r\beta(U)}{=}\int_{{\mathbf{S}}_{\Vert\cdot\Vert}^{K}}\int
_{0}^{\infty}\frac{h(s,U)}{\beta(U)}\left(  1-e^{-r\beta(U)/h(s,U)}\right)
\mathrm{d}s\alpha(\mathrm{d}U)\nonumber\\
&  +\int_{{\mathbf{S}}_{\Vert\cdot\Vert}^{K}}\int_{0}^{\infty}E_{1}\left(
\frac{h(s,U)}{\beta(U)}\right)  \mathrm{d}s\alpha(\mathrm{d}U) \label{Imp}%
\end{align}
where $E_{1}$ is the exponential integral function as in the proof of
Proposition \ref{ExitMulGam}.

Using the change of variable $z=\beta(U)/h(s,U)$ we have%

\begin{align}
I  &  =\int_{{\mathbf{S}}_{\Vert\cdot\Vert}^{K}}\int_{0}^{\infty}%
\frac{1-e^{-z}}{z}G_{U}(\mathrm{d}z)\alpha(\mathrm{d}U)\nonumber\\
&  +\int_{{\mathbf{S}}_{\Vert\cdot\Vert}^{K}}\int_{0}^{\infty}E_{1}\left(
z\right)  G_{U}(\mathrm{d}z)\mathrm{d}s\alpha(\mathrm{d}U)=I_{1}+I_{2}
\label{ca0}%
\end{align}
We shall in the following show that $I<\infty$ if and only if
\begin{equation}
I_{3}=\int_{{\mathbf{S}}_{\Vert\cdot\Vert}^{K}}\int_{0}^{\infty}\ln\left(
1+\frac{1}{z}\right)  G_{U}(\mathrm{d}z)\alpha(\mathrm{d}U)<\infty\label{ca1}%
\end{equation}
if and only if (\ref{CondLebMeaha}) and (\ref{CondLebMeahb}) are satisfied.
This concludes, as obviously \eqref{ca1} and $h\in L(\Gamma_{K}(\alpha
,\beta))$ are equivalent.

First,
\begin{equation}
\int_{{\mathbf{S}}_{\Vert\cdot\Vert}^{K}}\int_{1/2}^{\infty}\frac{1-e^{-z}}%
{z}G_{U}(\mathrm{d}z)\alpha(\mathrm{d}U)<\infty\label{ca3}%
\end{equation}
if and only if (\ref{CondLebMeahb}), if and only if
\[
I_{4}=\int_{{\mathbf{S}}_{\Vert\cdot\Vert}^{K}}\int_{1/2}^{\infty}\ln\left(
1+\frac{1}{z}\right)  G_{U}(\mathrm{d}z)\alpha(\mathrm{d}U)<\infty
\]
since $(1/z)/\ln(1+z^{-1})\rightarrow1$ as $z\rightarrow\infty.$

On the other hand,
\[
\int_{{\mathbf{S}}_{\Vert\cdot\Vert}^{K}}\int_{0}^{1/2}E_{1}\left(  z\right)
G_{U}(\mathrm{d}z)\alpha(\mathrm{d}U)<\infty
\]
if and only if (\ref{CondLebMeaha}) holds (since $E_{1}(z)/(-\ln
(z))\rightarrow1$ as $z\rightarrow0$) if and only if
\[
I_{5}=\int_{{\mathbf{S}}_{\Vert\cdot\Vert}^{K}}\int_{0}^{1/2}\ln\left(
1+\frac{1}{z}\right)  G_{U}(\mathrm{d}z)\alpha(\mathrm{d}U)<\infty
\]
because $\ln(1+z^{-1})/(-\ln(z))\rightarrow1$ as $z\rightarrow0$.

$I_{5}<\infty$ and \eqref{CondLebMeaha} both imply
\[
\int_{{\mathbf{S}}_{\Vert\cdot\Vert}^{K}}\int_{0}^{1/2}G_{U}(\mathrm{d}%
z)\alpha(\mathrm{d}U)<\infty.
\]
Thus%
\[
\int_{{\mathbf{S}}_{\Vert\cdot\Vert}^{K}}\int_{0}^{1/2}\frac{1-e^{-z}}{z}%
G_{U}(\mathrm{d}z)\alpha(\mathrm{d}U)\leq\int_{{\mathbf{S}}_{\Vert\cdot\Vert
}^{K}}\int_{0}^{1/2}G_{U}(\mathrm{d}z)\alpha(\mathrm{d}U)<\infty,
\]
provided $I_{5}<\infty$ or \eqref{CondLebMeaha} hold.

Since $0\leq E_{1}(z)\leq e^{-z}\ln(1+1/z)\,\forall\,z\in\mathbb{R}_{+}$ (see
\cite[p. 229]{AbramowitzStegun}), using (\ref{ca3}) implies
\begin{align*}
&  \int_{{\mathbf{S}}_{\Vert\cdot\Vert}^{K}}\int_{1/2}^{\infty}E_{1}\left(
z\right)  G_{U}(\mathrm{d}z)\mathrm{d}s\alpha(\mathrm{d}U)\\
&  \leq\int_{{\mathbf{S}}_{\Vert\cdot\Vert}^{K}}\int_{1/2}^{\infty}e^{-z}%
\ln\left(  1+\frac{1}{z}\right)  G_{U}(\mathrm{d}z)\alpha(\mathrm{d}U)\leq
e^{-1/2} I_{4}.
\end{align*}

\end{proof}

\begin{proposition}
\label{ExitInt2} Let $h\in L(\Gamma_{K}(\alpha,\beta))$. Then the distribution
of the $K$-valued random variable $Y^{h}$ is infinitely divisible and has
Laplace transform
\begin{align}
\emph{L}_{Y^{h}}(\Theta)  &  =\exp\left(  -\int_{{\mathbf{S}}_{\Vert\cdot
\Vert}^{K}}\int_{0}^{\infty}\int_{0}^{\infty}\left(  1-e^{-rh(t,U)\left<
\Theta,U\right>  }\right)  \frac{e^{-\beta(U)r}}{r}\mathrm{d}t\mathrm{d}%
r\alpha(\mathrm{d}U)\right) \label{LTYh1}\\
&  =\exp(-\int_{{\mathbf{S}}_{\Vert\cdot\Vert}^{K}}\int_{0}^{\infty}\ln\left(
1+\frac{\left<  \Theta,U\right>  }{z}\right)  G_{U}(\mathrm{d}z)\alpha
(\mathrm{d}U).
\end{align}
where for $\alpha$-$a.e$. $U,$ $G_{U}$ is a Thorin measure \textit{measure} on
$\mathbb{R}_{+}$ which is the image of Lebesgue measure on $\mathbb{R}_{+}$
under the change of variable $s\rightarrow\beta(U)/h(s,U).$ Moreover, the Lévy
measure of $Y^{h}$ is
\begin{equation}
\nu_{Y^{h}}(E)=\int_{{\mathbf{S}}_{\Vert\cdot\Vert}^{K}}\int_{0}^{\infty}%
1_{E}(rU)\frac{k_{U}(r)}{r}\mathrm{d}r\alpha(\mathrm{d}U),\quad E\in
\mathcal{B}(K) \label{LevMeaYh}%
\end{equation}
where
\begin{equation}
k_{U}(r)=\int_{0}^{\infty}e^{-rz}G_{U}(\mathrm{d}z). \label{LevMeaYhk}%
\end{equation}

\end{proposition}

\begin{proof}
Using the obvious analogue for the Laplace transform of the formulae for the
characteristic functions of the integrals with respect to i.d.i.s.r.m.s in
\cite{RaRo89,BnSe09b} we obtain
\begin{align*}
\emph{L}_{Y^{h}}(\Theta)  &  =\exp\left(  -\int_{K}\int_{0}^{\infty}\left(
1-e^{-\left<  \Theta,h(s,\frac{x}{\left\Vert x\right\Vert } )x\right>
}\right)  \nu_{\mu_{\alpha,\beta}}(\mathrm{d}x)\mathrm{d}s\right) \\
&  =\exp\left(  \int_{0}^{\infty}\int_{{\mathbf{S}}_{\Vert\cdot\Vert}^{K}}%
\int_{0}^{\infty}\left(  1-e^{-rh(s,U)\left<  \Theta,U\right>  }\right)
\frac{e^{-r\beta(U)}}{r}\mathrm{d}r\alpha(\mathrm{d}U)\mathrm{d}s\right)  .
\end{align*}
As in the last proposition, let $G_{U}(\mathrm{d}z)$ be the \textit{measure}
on $\mathbb{R}_{+}$ which is the image of the Lebesgue measure on
$\mathbb{R}_{+}$ under the change of variable $s\rightarrow\beta
(U)/h(s,U)=z(U)$. Then%
\begin{align*}
\emph{L}_{Y^{h}}(\Theta)  &  =\exp\left(  \int_{0}^{\infty}\int_{{\mathbf{S}%
}_{\Vert\cdot\Vert}^{K}}\int_{0}^{\infty}\left(  1-e^{-r\left<  \Theta
,U\right>  }\right)  \frac{e^{-r\beta(U)/h(s,U)}}{r}\mathrm{d}r\alpha
(\mathrm{d}U)\mathrm{d}s\right) \\
&  =\exp\left(  \int_{{\mathbf{S}}_{\Vert\cdot\Vert}^{K}} \int_{0}^{\infty
}\int_{0}^{\infty}\left(  1-e^{-r\left<  \Theta,U\right>  }\right)
\frac{e^{-rz}}{r}G_{U}(dz)\mathrm{d}r\alpha(\mathrm{d}U)\right) \\
&  =\exp\left(  \int_{{\mathbf{S}}_{\Vert\cdot\Vert}^{K}} \int_{0}^{\infty
}\left(  1-e^{-r\left<  \Theta,U\right>  }\right)  \frac{k_{U}(r)}%
{r}\mathrm{d}r\alpha(\mathrm{d}U)\right)  .
\end{align*}
Hence, combining (\ref{tmGGC}) with the existence conditions for the integral,
$G_{U}(\mathrm{d}t)$ is a Thorin measure on $\mathbb{R}_{+}$ for $\alpha
$-$a.e$. $U$.
\end{proof}

\subsection{Characterisation of Cone Valued GGC}

In this section we define Generalized Gamma Convolutions $GGC$($K$) in the
cone $K$ and characterize this class as the distributions of the $K$-valued
random elements represented by the stochastic integral (\ref{wginrep}). The
result is an extension to the cone valued case of the Wiener-Gamma integral
representation of one-dimensional generalised Gamma convolutions, see Section
1.2 in \cite{JRY08}.

Similar to the multivariate case (see \cite{BMS06}), we define $GGC$($K$) as follows

\begin{definition}
The class $GGC(K)$ is the collection of all infinitely divisible distributions on
$K$ with Lévy measure $v_{\mu}$ having a polar decomposition
\begin{equation}
\nu_{\mu}(E)=\int_{{\mathbf{S}}_{\Vert\cdot\Vert}^{K}}\int_{0}^{\infty}%
1_{E}(rU)\frac{k_{U}(r)}{r}\mathrm{d}r\alpha(\mathrm{d}U),\quad E\in
\mathcal{B}(K), \label{LVGGCK}%
\end{equation}
where $k_{U}(r)$ is a measurable function in $U$ and completely monotone in
$r$ for $\alpha$-a.e. $U.$ 
\end{definition}
A probabilistic interpretation of the class $GGC(K)$
is provided by Theorem \ref{ERVCGGCK} below.

Proposition \ref{ExitInt2} says that the class of distributions of the
Wiener-Gamma integrals $Y_{h}$, $h\in L(\Gamma_{K}(\alpha,\beta))$, are $GGC$%
($K$). 

We now prove that all distributions in $GGC$($K$) have a Wiener-Gamma integral representation.
For simplicity we consider the case without drift, that is
$\Psi_{0}=0$ in (\ref{gllkr}) and (\ref{itorep}). Otherwise%
\[
Y^{h}=\Psi_{\mathrm{0}}+\int_{0}^{\infty}\int_{K}h\left(s,\frac{x}{\left\Vert
x\right\Vert }\right)xN(\mathrm{d}s,\mathrm{d}x).
\]

\begin{theorem}[Wiener-Gamma characterization of $GGC$($K$)]
\label{WGCGGCK} For any fixed
Borel-meas\-ur\-able function $\beta: \mathbf{S}^{K}_{\|\cdot\|}\to\mathbb{R}_{+}$
bounded away from zero it holds that
\begin{align*}
&\left\{  Y^{h}=\int_{0}^{\infty}\int_{K}h(s,x/\|x\|)xN(\mathrm{d}s,\mathrm{d}%
x):\,\,\alpha\mbox{ a finite measure on } \mathcal{B}(\mathbf{S}^{K}%
_{\|\cdot\|}),h\in L(\Gamma_{K}(\alpha,\beta))\right\} \\&\quad\quad =GGC_0(K)
\end{align*}
with $GGC_0(K)$ denoting all generalized Gamma convolutions on $K$ without drift.
\end{theorem}

The condition on $\beta$ above is needed to ensure the existence of the Gamma
random variables for all finite measures $\alpha$. The result implies that starting with any fixed homogeneous (or non-homogeneous with $\beta$ bounded away from 0) Gamma distribution one can obtain all generalized Gamma convolutions as the sum of a constant and a Wiener-It\^o integral with respect to the jump measure obtained from this fixed distribution.

\begin{proof}
Let $\mu\in GGC_0(K)$ with Lévy measure given by \ (\ref{LVGGCK}). Since
$k_{U}(r)$ is completely monotone in $r$ for $\alpha$-a.e. $U$, there exists a
Radon measure $G_{U}$ such that $k_{U}(r)=\int_{0}^{\infty}e^{-rz}%
G_{U}(\mathrm{d}z)$. Moreover%
\begin{equation}
\int_{{\mathbf{S}}_{\Vert\cdot\Vert}^{K}}\int_{0}^{\infty}\min(1,r)\frac
{k_{U}(r)}{r}\mathrm{d}r\alpha(\mathrm{d}U)<\infty. \label{WGCGGCK1}%
\end{equation}

Let $F_{G_{U}}(x)=\int_{0}^{x}G_{U}($\textrm{d}$z)$ for $x\geq0$ and
$F_{G_{U}}^{-1}(s)$ be the right continuous generalised inverse of $F_{G_{U}%
}(s)$. Let $\tilde h(s,U)=1/F_{G_{U}}^{-1}(s)$ and $h(s,U)=\beta(U)\tilde h (s,U)$ for $s\geq0$. It follows as in the one
dimensional case that $G_{U\mu}$, $\alpha$-a.e. $U$, is a Thorin measure which
is the image of Lebesgue measure on $(0,\infty)$ under the change of variable
$s\rightarrow1/h(s,U)$. That is,
\[
\int_{0}^{\infty}e^{-\frac{\beta(U)x}{h(s,U)}}\mathrm{d}s=\int_{0}^{\infty}e^{-\frac{x}{\tilde h(s,U)}}\mathrm{d}s=\int_{0}^{\infty}%
e^{-xz}G_{U}(\mathrm{d}z),\quad x>0
\]
and
\begin{align*}
I  &  =\int_{{\mathbf{S}}_{\Vert\cdot\Vert}^{K}}\int_{0}^{\infty}\int
_{0}^{\infty}\frac{\min(1,r)}{r}e^{-rz}G_{U}(\mathrm{d}z)\mathrm{d}r\alpha
(\mathrm{d}U)\\
&  =\int_{{\mathbf{S}}_{\Vert\cdot\Vert}^{K}}\int_{0}^{\infty}\int_{0}^{\infty}\frac{h(s,U)}{\beta(U)}\left(
1-e^{-\beta(U)r/h(s,U)}\right)  dr\mathrm{d}s\alpha(\mathrm{d}U)\\
&  +\int_{{\mathbf{S}}_{\Vert\cdot\Vert}^{K}}\int_{0}^{\infty}E_{1}\left(\frac{
h(s,U)}{\beta(U)}\right)  \mathrm{d}s\alpha(\mathrm{d}U).
\end{align*}
Thus \eqref{WGCGGCK1}, Proposition \ref{ExitInt1} and \eqref{Imp}  imply  $h\in L(\Gamma_{K}(\alpha,\beta))$. Let $N(\mathrm{d}s,\mathrm{d}x)$
be the Poisson random measure associated to $\Gamma_{K}(\alpha,\beta)$ and
$Y^{h}=\int_{0}^{\infty}\int_{K}h(s,x/\left\Vert x\right\Vert )xN(\mathrm{d}%
s,\mathrm{d}x)$. Then Proposition \ref{ExitInt2} shows that $\mu$ is the distribution of $Y^h$ which concludes the proof.
\end{proof}

We also have another characterization of $GGC(K)$, similar to a
characterization of multivariate GGC proved in \cite[Theorem F]{BMS06}. This
gives another probabilistic interpretation of $GGC(K)$.

We call $XV$ an elementary Gamma variable in $K$ if $X$ is a non-random
non-zero vector in $K$ and $V$ is a non-negative real random variable with
one-dimensional Gamma distribution $\Gamma(\alpha,\beta)$.

\begin{theorem}
\label{ERVCGGCK} \ $GGC$($K$) is the smallest class of distributions on $K$
closed under convolution and weak convergence and containing the distributions
of all elementary Gamma variables in $K$.
\end{theorem}

\begin{proof}
The proof is along the same lines to that of Theorem $F$ in \cite[p.
27]{BMS06}.\ 
\end{proof}

This implies also that GGC($K$) is the smallest class of distributions closed
under convolution and weak convergence containing all $K$-valued $\Gamma$
distributions in the sense of Definition \ref{def:gammaconvalued}. It is
trivial to see that GGC($K$) includes all $K$-valued stable distribution using
the spectral representation of stable L\'evy measures and maps $h$ of the form
$h(s,U)=\{s\theta(U) \Gamma(\alpha+1)\}^{-\frac{1}{\alpha}}$ with $0<\alpha<1$.

\section{Positive Definite Matrix Gamma Distributions}

\label{sec:matrix}

\label{sec:conegammadist}

In this section we consider the important case of non-negative definite Gamma
random matrices. This corresponds to the closed cone $K=\overline{\mathbb{S}%
}_{d}^{+}$ of symmetric non-negative definite $d\times d$ matrices with inner
product $\left\langle X,Y\right\rangle =\mathrm{tr}(X^{\top}Y)$,
$X,Y\in\overline{\mathbb{S}}_{d}^{+}$. \ When $X$ is in the open cone
$\mathbb{S}_{d}^{+}$, we write $X>0$. When dealing with random matrices, a
useful matrix norm is the \textit{trace norm} defined for $X\in\mathbb{M}%
_{d}(\mathbb{R})$ as $\left\Vert X\right\Vert =\mathrm{tr}(\left(  X^{\top
}X\right)  ^{1/2}).$ We write $\mathbf{S}_{\left\Vert \cdot\right\Vert }%
^{+}=\mathbf{S}_{\left\Vert \cdot\right\Vert }^{\overline{\mathbb{S}}_{d}^{+}%
}=\mathbf{S}_{\left\Vert \cdot\right\Vert }\cap\overline{\mathbb{S}}_{d}^{+}$.
For $X\in\overline{\mathbb{S}}_{d}^{+},$ $\left\Vert X\right\Vert
=\mathrm{tr}(X)$ and, in particular, if $U\in\mathbf{S}_{\left\Vert
\cdot\right\Vert }^{+}$, $\mathrm{tr}(U)=\left\Vert U\right\Vert =1.$ By
Proposition 3.4 it is not important which norm we use.\ So we choose the one
most convenient to work with.

\subsection{General Case}

The matrix Gamma distribution $\mu\sim$ $\Gamma_{\overline{\mathbb{S}}_{d}%
^{+}}(\alpha,\beta)$ on $\overline{\mathbb{S}}_{d}^{+}$ \ has the Laplace
transform
\begin{equation}
L_{\mu}(\Theta)=\exp\left\{  -\int_{\mathbf{S}_{\left\Vert \cdot\right\Vert
}^{+}}\int_{0}^{\infty}\left(  1-e^{-r\mathrm{tr}(\Theta U)}\right)
\frac{e^{-\beta(U)r}}{r}\mathrm{d}r\alpha(\mathrm{d}U)\right\}  ,\forall
\Theta\in\overline{\mathbb{S}}_{d}^{+} \label{LTGamMat}%
\end{equation}
with alternative representation
\begin{equation}
L_{\mu}(\Theta)=\exp\left\{  -\int_{\mathbf{S}_{\left\Vert \cdot\right\Vert
}^{+}}\ln\left(  1+\frac{\mathrm{tr}(U\Theta)}{\beta(U)}\right)
\alpha(\mathrm{d}U)\right\}  ,\forall\Theta\in\overline{\mathbb{S}}_{d}^{+}.
\end{equation}

Additional properties of Gamma random matrices to those for the general cone
valued case in\ Section 4.2 are the following. \ 

If $M$ is a symmetric random matrix with Gamma distribution $\Gamma
_{\overline{\mathbb{S}}_{d}^{+}}(\alpha,\beta)$, $\mathrm{tr}(M)$ follows a
one-dimensional Gamma convolution law. However, in the homogeneous case
$\beta(U)=\beta_{0}>0$, $\mathrm{tr}(M)$ \ has a one-dimensional \ Gamma
distribution $\Gamma(\alpha(\mathbf{S}_{\left\Vert \cdot\right\Vert }%
^{+}),\beta_{0})$.

\begin{proposition}
\label{TrDis} a) If $M\sim\Gamma_{\overline{\mathbb{S}}_{d}^{+}}(\alpha
,\beta)$, $\mathrm{tr}(M)$ has a one-dimensional $GGC$ law with Laplace
transform
\begin{equation}
\mathbb{E}e^{-\theta\mathrm{tr}(M)}=\exp\left\{  -\int_{0}^{\infty}\ln\left(
1+\frac{\theta}{s}\right)  \upsilon_{\alpha,\beta}(\mathrm{d}s)\right\}
\label{TrDis1}%
\end{equation}
where $\upsilon_{\alpha,\beta}$ is the Thorin measure on $(0,\infty)$ induced
by $\alpha(\mathrm{d}U)$ on $\mathbf{S}_{\left\Vert \cdot\right\Vert }^{+}$
under the transformation $U\rightarrow\beta(U).$

b) If $M\sim\Gamma_{\overline{\mathbb{S}}_{d}^{+}}(\alpha,\beta)$ with
$\beta(U)=\beta_{0}$, $\mathrm{tr}(M)$ has the one-dimensional Gamma
distribution $\Gamma(\alpha(\mathbf{S}_{\left\Vert \cdot\right\Vert }%
^{+}),\beta_{0}).$
\end{proposition}

\begin{proof}
For $\theta>0$, let $\Theta=\theta\mathrm{I}_{d}$. Since $\mathbb{E}%
e^{-\theta\mathrm{tr}(M)}=\emph{L}_{\mu}(\Theta)$, from (\ref{AltLTMat})
\begin{align*}
\mathbb{E}e^{-\theta\mathrm{tr}(M)}  &  =\exp\left\{  -\int_{\mathbf{S}%
_{\left\Vert \cdot\right\Vert }^{+}}\ln\left(  1+\frac{\theta\mathrm{tr}%
(U)}{\beta(U)}\right)  \alpha(\mathrm{d}U)\right\} \\
&  =\exp\left\{  -\int_{0}^{\infty}\ln\left(  1+\frac{\theta}{s}\right)
\upsilon_{\alpha,\beta}(\mathrm{d}s)\right\}
\end{align*}
where $\upsilon_{\alpha,\beta}$ is the measure on $(0,\infty)$ induced by
$\alpha(\mathrm{d}U)$ on $\mathbf{S}_{\left\Vert \cdot\right\Vert }^{+}$ under
the transformation $U\rightarrow\beta(U)$. Then, using (\ref{cfGGC}) we obtain
(a). For (b) we observe that from the first equality in the last expression
with $\beta(U)=\beta_{0},$ we obtain $\mathbb{E}e^{-\theta\mathrm{tr}%
(M)}=\left(  1+\theta/\beta_{0}\right)  ^{-\alpha(\mathbf{S}_{\left\Vert
\cdot\right\Vert }^{+})}.$
\end{proof}

Any matrix Gamma distribution $\Gamma_{\overline{\mathbb{S}}_{d}^{+}}%
(\alpha,\beta)$ \ is self-decomposable and if $\operatorname{supp}({\alpha)}$
is of full dimension, it is absolutely continuous with respect to the Lebesgue
measure on $\mathbb{S}_{d}$ (which can be identified with $\mathbb{R}%
^{d(d+1)/2}$) and so there is a density. \ The proof follows from the
multivariate case, identifying the cone $\mathbb{S}_{d}^{+}$ with
$\mathbb{R}^{d(d+1)/2}.$ Moreover, since the Lebesgue measure of
$\overline{\mathbb{S}}_{d}^{+}\backslash\mathbb{S}_{d}^{+}$ is zero, the
distribution $\Gamma_{\overline{\mathbb{S}}_{d}^{+}}(\alpha,\beta)$ is
supported in the the open cone $X>0.$ In other words

\begin{corollary}
\label{PDP1} Let $M$ be a random matrix with Gamma distribution $\Gamma
_{\overline{\mathbb{S}}_{d}^{+}}(\alpha,\beta)$ \ with $\operatorname{supp}%
(\alpha)$ of full dimension. Then $\mathbb{P}(M>0)=1.$
\end{corollary}

The following result is an adaptation of Proposition 3.8 to special linear
operators preserving the cone $\mathbb{S}_{d}^{+}$. Observe that all
invertible surjective linear operators preserving $\mathbb{S}_{d}^{+}$ are of
the form $X\mapsto CXC^{T}$ with some $C\in GL_{d}(\mathbb{R})$ (see
\cite{Lietal2001,Loewy1992}).

\begin{proposition}
\label{Ma6} Let $M$ $\sim$ $\Gamma_{\overline{\mathbb{S}}_{d}^{+}}%
(\alpha,\beta)$ with respect to the trace norm $\left\Vert \cdot\right\Vert $
and let $C\in$ $\mathcal{GL}_{d}(\mathbb{R})$. Then $Y=CMC^{\top}\sim$
$\Gamma_{\mathbf{S}_{\left\Vert \cdot\right\Vert _{C}}^{+}}(\alpha_{c}%
,\beta_{c}),$ where $\mathbf{S}_{\left\Vert \cdot\right\Vert _{C}}%
^{+}=\mathbf{S}_{\left\Vert \cdot\right\Vert _{c}}\cap\overline{\mathbb{S}%
}_{d}^{+}$ \ for $\left\Vert B\right\Vert _{C}=\left\Vert C^{-1}BC^{-\top
}\right\Vert $ and
\begin{equation}
\alpha_{C}(E)=\alpha(C^{-1}EC^{-\top})\text{,\quad}\forall E\in\mathcal{B}%
(\mathbf{S}_{\left\Vert \cdot\right\Vert _{C}}^{+})\text{,} \label{alfaLT}%
\end{equation}
and%
\begin{equation}
\beta_{C}(U)=\beta(C^{-1}EC^{-\top})\text{,\quad}\forall U\in\mathbf{S}%
_{\left\Vert \cdot\right\Vert _{C}}^{+}. \label{betaLT}%
\end{equation}
Moreover, $Y\sim$ $\Gamma_{\overline{\mathbb{S}}_{d}^{+}}(\widehat{\alpha}%
_{C},\widehat{\beta}_{C})$ with respect to the trace norm $\left\Vert
\cdot\right\Vert $ where
\begin{equation}
\widehat{\alpha}_{C}(E)=\int_{\mathbf{S}_{d}^{+}(C)}1_{E}\left(  \frac
{U}{\Vert U\Vert}\right)  \alpha_{C}(\mathrm{d}U),\quad\forall\,E\in
\mathcal{B}(\mathbf{S}_{\Vert\cdot\Vert}^{+}) \label{alfaLT1}%
\end{equation}
and%
\begin{equation}
\widehat{\beta}_{C}(U)=\beta_{C}\left(  \frac{U}{\Vert U\Vert_{C}}\right)
\Vert U\Vert_{C}\,,\quad\forall U\in\mathbf{S}_{\Vert\cdot\Vert}^{+}.
\label{betaLT1}%
\end{equation}

\end{proposition}

\begin{example}
\textbf{ }\label{Ma1} (Diagonal matrix with independent entries). As pointed
out in \cite{BNPA08}, an infinitely divisible non-negative definite random
matrix $M$ has independent components if and only if it is diagonal and
therefore its Lévy measure is concentrated in the diagonal matrix axes
$E^{ii}\in\overline{\mathbb{S}}_{d}^{+}$, $i=1,...d.$ Thus, a Gamma random
matrix $M$ $\sim$ $\Gamma_{\overline{\mathbb{S}}_{d}^{+}}(\alpha,\beta)$ has
independent components, if and only if there exist non-negative numbers
$\beta_{1},...,\beta_{d}$ such that the Lévy measure $\nu_{X}$ is given by
\[
\nu_{M}(E)=%
{\displaystyle\sum\limits_{i=1}^{d}}
\alpha(E_{ii})\int_{0}^{\infty}1_{E}(rU)\frac{e^{-r\beta_{i}}}{r}%
\mathrm{d}r\quad E\in\mathcal{B}(\overline{\mathbb{S}}_{d}^{+}).
\]

\end{example}

Further examples are considered in the next section.

\subsection{The A$\Gamma$-distribution}

We now introduce a special matrix distribution $A\Gamma_d(\eta,\Sigma)$ in the
open cone $\mathbb{S}_{d}^{+}$ with parameters, $\eta>(d-1)/2$ and $\Sigma
\in\mathbb{S}_{d}^{+}$. \ We study several properties including a relation
between cumulants of $A\Gamma(\eta,\Sigma)$ and the moments of a Wishart
distribution. 

The multivariate Gamma function, denoted by $\Gamma_{d}(\eta)$, is defined for
$\operatorname{Re}(\eta)>(d-1)/2$ as
\begin{equation}
\Gamma_{d}(\eta)=\int_{X>0}e^{-\mathrm{tr(X})}\left\vert X\right\vert
^{\eta-(d+1)/2}\mathrm{d}X, \label{MulGammFunc}%
\end{equation}
where \textrm{d}$X$ is the Lebesgue measure on $\mathbb{S}_{d}^{+}$
(identified with $\mathbb{R}^{d(d+1)/2})$; see for example \cite[p.61]%
{Mu82}$.$ An alternative expression for $\Gamma_{d}(\eta)$ is (\cite[Theorem
2.1.12]{Mu82})
\begin{equation}
\Gamma_{d}(\eta)=\pi^{d(d-1)/4}%
{\displaystyle\prod\limits_{i=1}^{d}}
\Gamma\left(  \eta-\frac{1}{2}(i-1)\right)  ,\quad\operatorname{Re}%
(\eta)>(d-1)/2. \label{AltMulGammFunc}%
\end{equation}

The special infinitely divisible matrix Gamma distribution  $A\Gamma_d(\eta,\Sigma)$ is
defined as follows. For $\eta>(d-1)/2$, consider the measure $\rho_{\eta
}(\mathrm{d}X)=g_{\eta}(X)\mathrm{d}X$ on the open cone $\mathbb{S}_{d}^{+}$
where%
\begin{equation}
g_{\eta}(X)=c_{d,\eta}\frac{e^{-\mathrm{tr}(X)}}{\left(  \mathrm{tr}%
(X)\right)  ^{\eta d}}\left\vert X\right\vert ^{\eta-(d+1)/2},\quad
X>0,\label{LevDet}%
\end{equation}
and
\begin{equation}
c_{d,\eta}=\omega_{d,\eta}\frac{\Gamma(\eta d)}{\Gamma_{d}(\eta)}\label{Const}%
\end{equation}
and $\omega_{d,\eta}>0$ is given.

\begin{proposition}
\label{Ma2} Let $\eta>(d-1)/2.$ There exists a homogeneous Gamma matrix
distribution $\Gamma_{\mathbb{S}_{d}^{+}}(\alpha_{\eta},\beta)$ with respect
to the trace norm where $\beta(U)=1$ for each $U\in$ $\mathbf{S}_{\left\Vert
\cdot\right\Vert }^{+}$ and $\alpha_{\eta}$ is the measure on $\mathbf{S}%
_{\left\Vert \cdot\right\Vert }^{+}$ given by
\begin{equation}
\alpha_{\eta}(\mathrm{d}U)=c_{d,\eta}\left\vert U\right\vert ^{\eta}%
\frac{\mathrm{d}U}{\left\vert U\right\vert ^{(d+1)/2}} \label{LevyDetProbSph}%
\end{equation}
with $\alpha_{\eta}(\mathbf{S}_{\left\Vert \cdot\right\Vert }^{+}%
)=\omega_{d,\eta}$. Moreover, the Lévy measure of $\ \Gamma_{\mathbb{S}%
_{d}^{+}}(\alpha_{\eta},\beta)$ is $\rho_{\eta}$ and has a polar
decomposition
\begin{equation}
\rho_{\eta}(E)=\int_{E}g_{\eta}(X)\mathrm{d}X=c_{d,\eta}\int_{\mathbf{S}%
_{\left\Vert \cdot\right\Vert }^{+}}\int_{0}^{\infty}1_{E}(rU)\frac{e^{-r}}%
{r}\mathrm{d}r\frac{\left\vert U\right\vert ^{\eta}\mathrm{d}U}{\left\vert
U\right\vert ^{(d+1)/2}},\quad E\in\mathcal{B}(\overline{\mathbb{S}}_{d}^{+}).
\label{PolDecDet}%
\end{equation}

\end{proposition}

\begin{proof}
To show existence of the matrix distribution $\Gamma_{\mathbb{S}_{d}^{+}%
}(\alpha_{\eta},\beta)$, by Proposition 3.3 it suffices to prove that
$\alpha_{\eta}$ is a finite measure, since trivially $\beta$ satisfies
(\ref{MomConMat}). The fact that $\Gamma_{\mathbb{S}_{d}^{+}}(\alpha_{\eta
},\beta)$ is concentrated in the open cone $\mathbb{S}_{d}^{+}$ \ will follow
by Corollary \ref{PDP1} since from (\ref{LevyDetProbSph}) $\operatorname{supp}%
({\alpha}_{\eta})$ has full dimension.

For $X>0$ make the change of variable
\begin{equation}
X=rU,r=\mathrm{tr(}X),\text{ }\mathrm{tr(}U)=1,\mathrm{d}X=r^{d{(\mathrm{d}%
+1)/2}-1}\mathrm{d}r\mathrm{d}U \label{changvar}%
\end{equation}
(\cite[p. 111]{Ma97})$.$ Using this in (\ref{MulGammFunc}) and the fact that
$\left\vert rU\right\vert =r^{d}\left\vert U\right\vert $
\begin{align}
\Gamma_{d}(\eta)  &  =\int_{0}^{\infty}\int_{\mathbf{S}_{\left\Vert
\cdot\right\Vert }^{+}}r^{d\eta-1}e^{-r}\mathrm{d}r\left\vert U\right\vert
^{\eta}\frac{\mathrm{d}U}{\left\vert U\right\vert ^{(d+1)/2}}\nonumber\\
&  =\Gamma(\eta d)\int_{\mathbf{S}_{\left\Vert \cdot\right\Vert }^{+}%
}\left\vert U\right\vert ^{\eta}\frac{\mathrm{d}U}{\left\vert U\right\vert
^{(d+1)/2}} \label{aux1}%
\end{align}
and hence $\alpha_{\eta}(\mathbf{S}_{\left\Vert \cdot\right\Vert }^{+}%
)=\omega_{d,\eta}.$ Using again the change of variable (\ref{changvar}) we
have (\ref{PolDecDet}).
\end{proof}

\begin{definition}
Let $\eta>(d-1)/2$ and $\Sigma\in\mathbb{S}_{d}^{+}$. An infinitely divisible
$p\times p$ positive definite random matrix $M$ is said to follow the
distribution $A\Gamma_{d}(\eta,\Sigma)$ if it has Gamma distribution
$\Gamma_{\mathbb{S}_{d}^{+}}(\alpha_{\eta,\Sigma},\beta_{\Sigma})$ with
respect to the trace norm where $\beta_{\Sigma}(U)=$ $\mathrm{tr}(\Sigma
^{-1}U)$ and
\begin{equation}
\alpha_{\eta,\Sigma}(\mathrm{d}U)=\frac{1}{\left\vert \Sigma\right\vert
^{\eta}\mathrm{tr}(\Sigma^{-1}U)^{\eta d}}\alpha_{\eta}(\mathrm{d}U)
\label{MeaSphSig}%
\end{equation}
and $\alpha_{\eta}$ is given by (\ref{LevyDetProbSph}).
\end{definition}

\begin{remark}
a) The distribution $A\Gamma_{d}(\eta,\Sigma)$ has also as a parameter the
total mass $\alpha_{\eta}(\mathbf{S}_{\left\Vert \cdot\right\Vert }%
^{+})=\omega_{d,\eta}$. This parameter is conjectured to be of particular importance when considering the limiting spectral (eigenvalue)  distribution as the dimension goes to infinity, since it may then depend on $\eta$ or $d$. Particularly, interesting
choices of $\omega_{d,\eta}$ in this connection should be $d\eta$, $d$ or a constant.

b) The case $\eta=(d+1)/2$ was considered in Barndorff-Nielsen and Pérez-Abreu
\cite{BNPA08}.

c)  It is interesting to note that for $\eta\in((d-1)/2,(d+1)/2)$ the L\'evy density becomes infinity at the non-invertible elements of $\mathbb{S}_d^+$ (i.e. the matrices which are positive semi-definite, but not strictly), whereas for $\eta>(d+1)/2$ the L\'evy density becomes zero at the non-invertible elements of $\mathbb{S}_d^+$. For $\eta=(d+1)/2$ we have that $\alpha_{\eta,I_d}$ is the uniform measure on the unit sphere. This observation is also interesting in relation to applications using  e.g. $A\Gamma_d(\eta,\Sigma)$ matrix subordinators. If all jumps should really be strictly positive definite , then a model with $\eta>(d+1)/2$ should be appropriate, whereas one should use  $\eta\in((d-1)/2,(d+1)/2)$, if it seems desirable to have most of  the strictly positive definite jumps rather close to non-invertible matrices. The latter should be especially useful in stochastic volatility models in finance where very often news (resulting in jumps of the covariance) should only really affect a single asset or a special group of assets (like stocks of companies from the same country or the same branch of industry). 
\end{remark}

Note that if $M\sim A\Gamma(\eta,\mathrm{I}_{d})$, then $\Sigma^{1/2}%
M\Sigma^{1/2}\sim A\Gamma_{d}(\eta,\Sigma)$. This follows from Proposition
\ref{Ma6} which also gives together with Proposition \ref{Ma2}, that
$A\Gamma_{d}(\eta,\Sigma)$ has Lévy measure $\rho_{\eta,\Sigma}(\mathrm{d}%
X)=g_{\eta,\Sigma}(X)\mathrm{d}X$ where
\begin{equation}
g_{\eta,\Sigma}(X)=\frac{c_{d,\eta}}{\left\vert \Sigma\right\vert ^{\eta}%
}\frac{e^{-\mathrm{tr}(\Sigma^{-1}X)}}{\left[  \mathrm{tr}(\Sigma
^{-1}X)\right]  ^{\eta d}}\left\vert X\right\vert ^{\eta-(d+1)/2},\quad X>0.
\label{LevyDetFamA}%
\end{equation}
The existence of moments of all orders of $A\Gamma_{d}(\eta,\mathrm{I}_{d})$
follows since (\ref{MomConMat}) is trivially satisfied. The same is true for
$A\Gamma_{d}(\eta,\Sigma)$ since $\Sigma^{1/2}M\Sigma^{1/2}\sim A\Gamma
_{d}(\eta,\Sigma)$.

In the homogeneous case the distribution $A\Gamma_{d}(\eta,\sigma
\mathrm{I}_{d})$ with $\sigma\in\bbr^+$ is invariant under orthogonal conjugations and the trace
follows a one-dimensional Gamma distribution.

\begin{lemma}
\label{HomTr} Let $\eta>(d-1)/2$ and $\sigma>0.$

a) The distribution $A\Gamma_{d}(\eta,\sigma\mathrm{I}_{d})$ is
invariant under orthogonal conjugations.

b) If $M\sim A\Gamma_{d}(\eta,\sigma\mathrm{I}_{d})$, then $\mathrm{tr}(M)$
follows a one-dimensional Gamma distribution $\Gamma(\omega_{d,\eta},\sigma).$
\end{lemma}

\begin{proof}
It is well known that the measure $\mathrm{d}X/\left\vert X\right\vert
^{(d+1)/2}$ is invariant under the conjugation $X\rightarrow$ $CXC^{\top}$,
for $X>0$ and any non-singular matrix $C$ (see \cite[Example 6.19]{Ea83}). The
determinant and the trace norm functions are invariant under the conjugation
$X\rightarrow$ $OXO^{\top}$, for $X>0$ and any $O\in$ $\mathcal{O}%
(\mathrm{d})$. Thus the Lévy measure $\rho_{\eta,\sigma\mathrm{I}_{d}}$ with
density (\ref{LevyDetFamA}), $\Sigma=$ $\sigma\mathrm{I}_{d}$, is invariant
under orthogonal conjugations and so the matrix distribution \ $A\Gamma
(\eta,\sigma\mathrm{I}_{d})$ is. Proposition \ref{TrDis}(b) gives
(b).\smallskip\ 
\end{proof}

The cumulants of the distribution $A\Gamma_{d}(\eta,\sigma\mathrm{I}_{d})$ are
related to the moments of the Wishart distribution, as we prove below. Recall
that a $d\times d$ positive definite random matrix $W$ is said to have a
Wishart distribution $W_{d}(n,d)$ with parameters $n>d-1$ and $\Sigma\in$
$\mathbb{S}_{d}^{+},$ if its density function is given by
\begin{equation}
f_{W}(A)=\frac{1}{2^{\frac{dn}{2}}\Gamma_{d}(\frac{n}{2})\left\vert
\Sigma\right\vert ^{\frac{n}{2}}}e^{-\frac{1}{2}\mathrm{tr}(\Sigma^{-1}%
A)}\left\vert A\right\vert ^{(n-d-1)/2},\text{ }A>0. \label{WishDsty}%
\end{equation}
As usual we denote by $A\otimes B$ the tensor product of the matrices $A$ and
$B$. We recall that if $A$ and $B$ are in $\mathbb{M}_{n}$, then
$\mathrm{tr}(A\otimes B)=\mathrm{tr}(A)\mathrm{tr}(B)$ and $\left\vert
A\otimes B\right\vert =$ $\left\vert A\right\vert ^{n}\left\vert B\right\vert
^{n}$.

We use the notation \textrm{$B$}$(n,m)=\Gamma(n)\Gamma(m)/\Gamma(n+m)$ for the
Beta function, where $\operatorname{Re}(n)>0,\operatorname{Re}(m)>0$, and for
the multivariate Beta function, denoted by \textrm{$B$}$_{d}(x,y)$,
\begin{equation}
\mathrm{B}_{d}(n,m)=\frac{\Gamma_{d}(n)\Gamma_{d}(m)}{\Gamma_{d}(n+m)}%
,\quad\operatorname{Re}(n)>0,\operatorname{Re}(m)>0. \label{multBeta}%
\end{equation}

\begin{proposition}
\label{MainPropAG} Let $\eta>(d-1)/2$ and $g_{\eta,\Sigma}(X)$ be the Lévy
density of $A\Gamma_{d}(\eta,\Sigma).$ Let $W$ be a random matrix with Wishart
distribution $W_{d}(2\eta,\Sigma).$ For any integer $p>0$ the following three
identities hold:

a)
\begin{equation}
\int_{X>0}X^{p}g_{\eta,\Sigma}(X)\mathrm{d}X=\frac{\omega_{d,\eta}}{2^{p}%
}\mathrm{B}(\eta d,p)\mathbb{E}W^{p}. \label{main1}%
\end{equation}

b)%
\begin{equation}
\int_{X>0}X^{\otimes p}g_{\eta,\Sigma}(X)\mathrm{d}X=\frac{\omega_{d,\eta}%
}{2^{p}}\mathrm{B}(\eta d,p)\mathbb{E}(W^{\otimes p}). \label{main2}%
\end{equation}

c)
\begin{align}
\int_{X>0}\left\vert X\right\vert ^{p}g_{\eta,\Sigma}(X)\mathrm{d}X  &
=\frac{\omega_{d,\eta}}{2^{pd}}\mathrm{B}(\eta d,pd)\mathbb{E}(\left\vert
W\right\vert ^{p})\label{main3}\\
&  =\omega_{d,\eta}\Gamma_{d}(p)\frac{\mathrm{B}(\eta d,pd)}{\mathrm{B}%
_{d}(\eta,p)}\left\vert \Sigma\right\vert ^{p}.
\end{align}

\end{proposition}

\begin{proof}
The existence of the integrals (\ref{main1})-(\ref{main3}) is seen as follows.
The finiteness of the $p$-th moment of $A\Gamma_{d}(\eta,\Sigma)$ is equivalent to the
existence of the $p$-th moment of its Lévy measure (away from the origin). Then $\int_{X>0}\left\Vert
X\right\Vert ^{p}g_{\eta,\Sigma}(X)\mathrm{d}X<\infty$ for any $p>0$ gives the
existence of the integral in (\ref{main1}). Since for $X>0$
$\mathrm{tr}(X^{\otimes p})=\left(  \mathrm{tr}(X)\right)  ^{p}=\left\Vert
X\right\Vert ^{p}$ and $\left\vert X\right\vert \leq\left\Vert X\right\Vert
^{d}$, one also obtains the existence of the integrals in (\ref{main2}) and (\ref{main3}), respectively.

The identities in (a)-(c) are consequences of a more general result for
$q$-homo\-geneous functions which we now prove: Let $h:$ $\mathbb{S}_{d}^{+}\rightarrow
H$ be a function such that $h(rX)=r^{q}h(X)$ for any $r>0$ and $X\in
\mathbb{S}_{d}^{+}$ and some fixed $q>0$, where $H$ is $\mathbb{S}_{d}^{+}$, $\left(  \mathbb{S}%
_{d}^{+}\right)  ^{\otimes p}$ and $(0,\infty)$ for (a),(b) and (c)
respectively. In the cases (a) and (b) we have$\ q=p$ while for (c) $q=dp.$

The change of variable $V=\Sigma^{-1/2}X\Sigma^{-1/2},$ the invariance of the
measure  \linebreak $\left\vert X\right\vert ^{-(d+1)/2}\mathrm{d}X$ \ under non-singular
linear transformations and writing $h_{\Sigma}(Y)=\linebreak h(\Sigma^{1/2}Y\Sigma
^{1/2})$ give:
\begin{align*}
J  &  =\int_{X>0}h(X)g_{\eta,\Sigma}(X)\mathrm{d}X=c_{d,\eta}\int
_{X>0}h(X)\frac{c_{d,\eta}}{\left\vert \Sigma\right\vert ^{\eta}}%
\frac{e^{-\mathrm{tr}(\Sigma^{-1}X)}}{\left(  \mathrm{tr}(\Sigma
^{-1}X)\right)  ^{\eta d}}\left\vert X\right\vert ^{\eta-(d+1)/2}\mathrm{d}X\\
&  =c_{d,\eta}\int_{V>0}h_{\Sigma}(V)\frac{e^{-\mathrm{tr}(V)}}{\left(
\mathrm{tr}(V)\right)  ^{\eta d}}\left\vert V\right\vert ^{\eta-(d+1)/2}%
\mathrm{d}V.
\end{align*}
Using (\ref{changvar}), the definition of the Gamma function, first for
$\Gamma(q)$ and then for $\Gamma(\eta d+q)$, and (\ref{Const}) give
\begin{align}
J  &  =c_{d,\eta}\int_{0}^{\infty}\int_{\mathbf{S}_{\left\Vert \cdot
\right\Vert }^{+}}r^{q-1}h_{\Sigma}(U)e^{-r}\mathrm{d}r\left\vert U\right\vert
^{\eta}\frac{\mathrm{d}U}{\left\vert U\right\vert ^{(d+1)/2}}\nonumber\\
&  =c_{d,\eta}\Gamma(q)\int_{\mathbf{S}_{\left\Vert \cdot\right\Vert }^{+}%
}h_{\Sigma}(U)\left\vert U\right\vert ^{\eta}\frac{\mathrm{d}U}{\left\vert
U\right\vert ^{(d+1)/2}}\nonumber\\
&  =\frac{c_{d,\eta}\Gamma(q)}{\Gamma(\eta d+q)}\int_{\mathbf{S}_{\left\Vert
\cdot\right\Vert }^{+}}\int_{0}^{\infty}r^{\eta d+q-1}h_{\Sigma}%
(U)e^{-r}\mathrm{d}r\left\vert U\right\vert ^{\eta}\frac{\mathrm{d}%
U}{\left\vert U\right\vert ^{(d+1)/2}}\nonumber\\
&  =\frac{\omega_{d,\eta}\mathrm{B}(\eta d,q)}{\Gamma_{d}(\eta)}\int
_{V>0}h_{\Sigma}(V)e^{-\mathrm{tr}(V)}\left\vert V\right\vert ^{\eta
-(d+1)/2}\mathrm{d}V\label{aux5}\\
&  =\frac{\omega_{d,\eta}\mathrm{B}(\eta d,q)}{2^{\eta d}\Gamma_{d}(\eta)}%
\int_{A}h_{\Sigma}\left(\frac{1}{2}A\right)e^{-\frac{1}{2}\mathrm{tr}(A)}\left\vert
A\right\vert ^{\eta}\frac{\mathrm{d}A}{\left\vert A\right\vert ^{(d+1)/2}%
}\label{aux6}\\
&  =\omega_{d,\eta}\mathrm{B}(\eta d,q)\int_{A}h\left(\frac{1}{2}A\right)f_{W}%
(A)\mathrm{d}A=\omega_{d,\eta}\mathrm{B}(\eta d,q)\left(  \frac{1}{2}\right)
^{q}\mathbb{E}(h(W)), \label{aux8}%
\end{align}
where in (\ref{aux5}) we used again (\ref{changvar}), in (\ref{aux6}) the
change of variable $A=V/2$ (with $\mathrm{d}V=(1/2)^{(d+1)/2}\mathrm{d}A)$,
and for (\ref{aux8}) the fact that $f_{W}$ is the density (\ref{WishDsty}) of
the random matrix $W$ with Wishart distribution $W_{d}(2\eta$,$\Sigma)$.

Then (a), (b) and the first equality in (\ref{main3}) are proved. The second
equality in (\ref{main3}) follows using the fact that $\mathbb{E}\left\vert
W^{p}\right\vert =\left\vert \Sigma\right\vert ^{p}2^{dp}\Gamma_{d}%
(\eta+p)/\Gamma_{d}(\eta)$ (see Muirhead \cite[p. 101]{Mu82}).
\end{proof}

In particular, the mean $\mathbb{E}(M)$ and covariance $\mathrm{Cov(}%
M)=\mathbb{E}(M\otimes M)-\mathbb{E}(M)\otimes\mathbb{E}(M)$ of $M\sim
A\Gamma_{d}(\eta,\Sigma)$ are expressed in terms of the mean $\mathbb{E}(W)$
and the second tensor moment $\mathbb{E}(W)\otimes\mathbb{E}(W)$ of the
Wishart distribution $W_{d}(2\eta,\Sigma)$. Recall that the commutation
$d^{2}\times d^{2}$ matrix $K$ is defined as
\[
K=\sum_{i,j=1}^{d}H_{ij}\otimes H_{ij}^{\top}%
\]
where $H_{ij}$ denotes the $d\times d$ matrix with $h_{ij}=1$ and all other
elements zero. The $m$-th moments and cumulants of a $d\times d$ random matrix are $d^{2m}$-dimensional objects which need to be represented in a concise and at the same time easy to handle way. As usual for random matrices we define the moments and cumulants using the tensor product, e.g. the $m$-th moment of a random matrix $X$ is understood to be $E(X^{\otimes m})$. An alternative would be to use the $\operatorname{vec}$-operator to transfer the matrix into an element of $\bbr^{d^2}$ first, but typically this leads to formulae that are more cumbersome to handle.  Now we have:

\begin{corollary}
\label{CorMain1} The cumulants of the random matrix $M\sim A\Gamma_{d}%
(\eta,\Sigma)$ are proportional to the tensor moments of the Wishart
distribution. In particular
\begin{equation}
\mathbb{E}(M)=\frac{\omega_{d,\eta}}{d}\Sigma\label{FirstMoFamAOI}%
\end{equation}
and the matrix of covariances between elements of $M$ is given by
\begin{equation}
\mathrm{Cov}(M)=\omega_{d,\eta}\frac{\eta}{d(nd+1)}\left(\left(1+\frac{1}{2\eta
}\right)\mathrm{I}_{d^{2}}+K\right)(\Sigma\otimes\Sigma). \label{VarFamA}%
\end{equation}

\end{corollary}

\begin{proof}
The first assertion follows from (b) in Proposition \ref{MainPropAG}. Since
the first moment of $M$ equals its first cumulant and its matrix of
covariances equals its second cumulant, then (\ref{FirstMoFamAOI}) follows from
(a) in Proposition \ref{MainPropAG} with $p=1$ and since $\mathbb{E}%
(W)=2\eta\Sigma$ for $W\sim W_{d}(2\eta,\Sigma).$\ From \cite[p. 90]{Mu82} we
have
\[
\mathrm{Cov(}W)=2\eta(\mathrm{I}_{d^{2}}+K)(\Sigma\otimes\Sigma).
\]
Using (b) in Proposition \ref{MainPropAG} with $p=2$ we have
\begin{align*}
\mathrm{Cov(}M)  &  =\int_{X>0}X^{\otimes2}g_{\eta,\Sigma}(X)\mathrm{d}%
X=\frac{\omega_{d,\eta}}{4}\frac{1}{(nd+1)\eta d}\mathbb{E}(W^{\otimes2})\\
&  =\frac{\omega_{d,\eta}}{4}\frac{1}{(nd+1)\eta d}\left[  \mathrm{Cov(}%
W)+\left[  \mathbb{E}(W)\right]  ^{\otimes2}\right]  .
\end{align*}
Hence (\ref{VarFamA}) follows.
\end{proof}

In particular, when $\omega_{d,\eta}=d\eta$, $\mathbb{E}(M)=\eta\Sigma$, as in
the Wishart case. On the other hand, when $\omega_{d,\eta}=d$, $\mathbb{E}%
(M)=\Sigma$.

This result is of particular importance in applications, since it implies that the second order moment structure is explicitly known which may allow method of moments based estimation of models using  $A\Gamma_d(\eta,\Sigma)$ matrix subordinators as the stochastic input (e.g. \cite{PiSe09})

The following result states an interesting relation with the so-called
Marchenko-Pastur distribution of parameter $\lambda>0$. Recall that the
moments of this distribution are given by (see \cite{BS10})
\begin{equation}
\mu_{p}(\lambda)=\sum_{j=0}^{p-1}\frac{1}{j+1}\binom{p}{j}\binom{p-1}%
{j}\left(  \lambda\right)  ^{j}.\label{MPMoms}%
\end{equation}

\begin{lemma}
\label{CorMain3} Let $\varepsilon\in\mathbb{R}.$ For any integer $p>0,$ as $d\to\infty$ and
$d/\eta\rightarrow\lambda>0$
\begin{equation}
\frac{1}{d}\int_{X>0}\mathrm{tr}\left(  \frac{X}{d^{\varepsilon}}\right)
^{p}g_{\eta,\mathrm{I}_{d}}(X)\mathrm{d}X\sim K_{p}(\lambda)\omega_{d,\eta
}d^{-p(1+\varepsilon)}\label{CorMain3b}%
\end{equation}
where $K_{p}(\lambda)=\Gamma(p)\mu_{p}(2\lambda)$. In particular, for
$\varepsilon=1$
\begin{equation}
\lim_{d\rightarrow\infty}\frac{1}{d}\int_{X>0}\left[  \mathrm{tr}\left(
\frac{X}{d}\right)  ^{p}\right]  g_{\eta,\mathrm{I}_{d}}(X)\mathrm{d}X=%
\genfrac{\{}{.}{0pt}{}{\lambda\text{, if }p=1}{0\text{, if }p\geq2}%
.\label{CorMain3c}%
\end{equation}

\end{lemma}

\begin{proof}
If $W\sim W_{d}(2\eta,\mathrm{I}_{d})$, then \begin{equation}
\lim_{d\rightarrow\infty}\frac{1}{d}\mathbb{E}\mathrm{tr}\left(  \frac
{W}{2\eta}\right)  ^{p}=\mu_{p}(2\lambda).\label{MPconv}%
\end{equation} as $2\eta/d\rightarrow2\lambda$ due to the well known Marchenko-Pastur
Theorem \cite{BS10}, for any $p>0$,

By the Stirling approximation $\Gamma(z+1)\sim\sqrt{2\pi z}(z/e)^{z}$ for
$z\rightarrow\infty$, for $\eta$ and $d$ large%
\begin{equation}
\frac{\Gamma(\eta d)}{\Gamma(\eta d+p)}\sim(nd)^{-p}.\label{EstMom}%
\end{equation}
Using $\eta/d\rightarrow\lambda$ and (\ref{MPconv}) in (\ref{main1}) gives:%
\begin{align}
\frac{1}{d}\int_{X>0}\mathrm{tr}\left(  \frac{X}{d^{\varepsilon}}\right)
^{p}g_{\eta,\mathrm{I}_{d}}(X)\mathrm{d}X &  =\frac{\omega_{d,\eta}%
}{d^{\varepsilon p+1}2^{p}}B(\eta d,p)\mathbb{E}\mathrm{tr}(W^{p})\nonumber\\
&  =\frac{\omega_{d,\eta}}{d^{\varepsilon p}}\frac{\Gamma(\eta d)\Gamma
(p)}{\Gamma(\eta d+p)}\frac{1}{d}\mathbb{E}\mathrm{tr}(\frac{W}{2}%
^{p})\nonumber\\
&  \sim\Gamma(p)\omega_{d,\eta}(\eta d)^{-p}d^{-p\varepsilon}\left[  \frac
{1}{d}\mathbb{E}\mathrm{tr}\left(  \frac{W}{2}\right)  ^{p}\right]
\nonumber\\
&  \sim\Gamma(p)\omega_{d,\eta}d^{-p-p\varepsilon}\left[  \frac{1}%
{d}\mathbb{E}\mathrm{tr}\left(  \frac{W}{2}\right)  ^{p}\right]  \\
&  \sim K_{p}(\lambda)\omega_{d,\eta}d^{-p(1+\varepsilon)},\text{ for }%
\eta,d\text{ large,}\label{Cormain3a}%
\end{align}
which proves the lemma.
\end{proof}

\begin{conjecture}
We conjecture that the above Lemma is a first step to study the asymptotic
spectral distribution of the random matrix $M\sim A\Gamma_{d}%
(\eta,\Sigma)$. More specifically, the right hand side of (\ref{CorMain3c})
must be related to the $p$th-cumulant of the $p$th-moment of the mean spectral
distribution of $M$, which in turn should allow the identification of the limiting spectral distribution.
\end{conjecture}

\subsection{Further examples}

\subsubsection{B$\Gamma$-distributions}

Let $d\geq1$ and $q=1,...,d$ be fixed. Consider the Lévy measure on
$\overline{\mathbb{S}}_{d}^{+}$ given by
\begin{equation}
\nu_{q}(\mathrm{d}X)=\frac{e^{-\beta_{0}\left\Vert X\right\Vert }}{\left\Vert
X\right\Vert }\widetilde{\alpha}_{d,q}(\mathrm{d}X),\quad X\in\overline
{\mathbb{S}}_{d}^{+}\backslash\{\mathrm{0\}} \label{LMFamA}%
\end{equation}
where $\beta_{0}>0$ and
\begin{equation}
\widetilde{\alpha}_{d,q}(E)=\int_{\mathbf{S}_{\left\Vert \cdot\right\Vert
}^{+}}\int_{0}^{\infty}1_{E}(rU)\mathrm{d}r\alpha_{d,q}(\mathrm{d}U).
\label{LMFamAb}%
\end{equation}
Here $\alpha_{d,q}/d$ is the probability measure on the sphere $\mathbf{S}%
_{\left\Vert \cdot\right\Vert }^{+}$ induced by the transformation
$V\rightarrow U=VV^{\top}$, where the $d\times q$ matrix $V$ is uniformly
distributed on the unit sphere of$\ $the linear space $\mathbb{M}_{d\times
q}(\mathbb{R})$ of $d\times q$ matrices with real entries, with the Frobenius
norm $\left\Vert Y\right\Vert _{2}^{2}=\mathrm{tr}(Y^{\top}Y)$.

An infinitely divisible $d\times d$ symmetric random matrix $M$ with Lévy
measure $\nu_{q}$ has the Gamma distribution $\Gamma_{\overline{\mathbb{S}%
}_{d}^{+}}(\alpha_{d,q},\beta_{0})$, since $\nu_{q}$ has a polar
decomposition
\[
\nu_{q}(E)=\int_{\mathbf{S}_{\left\Vert \cdot\right\Vert }^{+}}\int
_{0}^{\infty}1_{E}(rU)\frac{e^{-r}}{r}\mathrm{d}r\alpha_{d,q}(\mathrm{d}%
U),\quad E\in\mathcal{B}(\overline{\mathbb{S}}_{d}^{+}).
\]
We call this distribution the $B\Gamma_d(q,\beta_0)$ distribution.

\begin{remark}
a) We observe that the support of $\nu_{q}$ is concentrated in matrices of
rank $q$ in $\overline{\mathbb{S}}_{d}^{+}$. Hence this support is of full
dimension. Then, by\ Corollary \ref{PDP1}$,$ $\Gamma_{\overline{\mathbb{S}%
}_{d}^{+}}(\alpha_{d,q},\beta_{0})$ has support in the open cone
$\mathbb{S}_{d}^{+}$.

b) The case $q=1$ was considered in Pérez-Abreu and Sakuma \cite{PAS08} in the
context of random matrix models for free generalised Gamma convolutions.
\ They considered the Hermitian case \ for which working in the setup of
$\ \mathbb{M}_{d\times q}(\mathbb{C})$ is needed, but otherwise the above
steps can be carried out in a straightforward way.

c) For $\Sigma\in\mathbb{S}_{d}^{+}$ one can consider invertible linear
transformations of $\ \Gamma_{\mathbb{S}_{d}^{+}}(\alpha_{d,q},\beta_{0})$ to
obtain infinitely divisible positive definite matrix Gamma distributions
$\Gamma_{\overline{\mathbb{S}}_{d}^{+}}(\alpha_{\eta,\Sigma},\beta_{\Sigma})$
with Lévy measures of the form
\[
\nu_{q}(\mathrm{d}X)=\frac{e^{-\left\Vert \Sigma^{-1}X\right\Vert }%
}{\left\Vert \Sigma^{-1}X\right\Vert }\widetilde{\alpha}_{d,q}(\Sigma
^{-1/2}\mathrm{d}X\Sigma^{-1/2})
\]
similar to the family $\Gamma$ of matrix Gamma distributions considered in the
last section.
\end{remark}

The following properties are easily proved.

\begin{proposition}
Let $M\sim$ $\Gamma_{\overline{\mathbb{S}}_{d}^{+}}(\alpha_{d,q},\beta_{0})$
and $q=1,...,d$ be fixed. Then,

a) $M$ has an invariant distribution under orthogonal conjugations.

b) $\mathbb{E}\left\Vert M\right\Vert ^{k}<\infty$ for any $k>0.$

d) $\mathrm{tr}(M)$ has a one-dimensional Gamma distribution $\Gamma
(d,\beta_{0}).$
\end{proposition}

\subsubsection{Matrix Gamma-Normal distribution}

In the one-dimensional case, the so called variance gamma distribution is
popular in applications in finance, see \cite{MadanCarrChang1998}. This
distribution is a mixture of Gaussians having a random variance following the
one-dimensional Gamma distribution. As an application of the matrix Gamma
distribution, we now present a matrix extension of the one-dimensional
variance Gamma distribution.

Let $Z$ be a $d\times q$ random matrix with independent standard Gaussian
distributed entries, i.e.
\[
\mathbb{E}\exp(i\mathrm{tr}(\Theta^{\top}Z))=\exp(-\frac{1}{2}\mathrm{tr}%
(\Theta^{\top}\Theta)),\quad\forall\Theta\in\mathbb{M}_{d\times q}%
(\mathbb{R}).
\]
Let $X$ be a random matrix with the Gamma distribution $\Gamma_{\overline
{\mathbb{S}}_{d}^{+}}(\alpha,\beta)$ and independent of $Z$. Consider the
\textit{random linear transformation }$Y=X^{1/2}Z.$ Using a standard
conditional argument we compute the characteristic function of the $d\times q$
matrix as follows:
\begin{align*}
\mathbb{E}\exp(i\mathrm{tr(}\Theta^{\top}Y))  &  =\mathbb{E}_{X}\mathbb{E}%
_{Z}\left[  \left.  \exp(i\mathrm{tr}(\Theta^{\top}X^{1/2}Z))\right\vert
X)\right] \\
&  =\mathbb{E}_{X}\left\{  \exp(-\frac{1}{2}\mathrm{tr}(X^{1/2}\Theta
\Theta^{\top}X^{1/2}))\right\} \\
&  =\mathbb{E}_{X}\left\{  \exp(-\frac{1}{2}\mathrm{tr}(\Theta\Theta^{\top
}X))\right\}  .
\end{align*}
Then, using (\ref{AltLTMat})%
\begin{equation}
\mathbb{E}\exp(i\mathrm{tr}(\Theta^{\top}Y))=\exp\left\{  -\int_{\mathbf{S}%
_{\left\Vert \cdot\right\Vert }^{+}}\ln\left(1+\frac{1}{2}\frac{\mathrm{tr}%
(U\Theta\Theta^{\top})}{\beta(U)}\right)\alpha(\mathrm{d}U)\right\}  ,
\label{ChFMatGGam}%
\end{equation}
for each $\Theta\in\mathbb{M}_{d\times q}(\mathbb{R}).$ Using the terminology
in \cite{BNPAR05}, we can say that $Y$ has a Mat$G$ distribution, which is
infinitely divisible in $\mathbb{M}_{d\times q}(\mathbb{R})$.

Similar to the one-dimensional case, we call this distribution the matrix
Gamma-Normal distribution with parameters $\alpha$ and $\beta$ or more
specifically the $d\times q$-dimensional matrix $\Gamma_{\overline{\mathbb{S}}_{d}^{+}}(\alpha,\beta
)$-Normal distribution. We observe that $Y$ has a symmetric distribution in
the sense that $-Y\overset{law}{=}Y$ and also that $Y$ has a distribution invariant under
orthogonal conjugations if $\beta(U)=\beta_{0}$ and $\alpha(\mathrm{d}U)$ is
invariant under orthogonal conjugations.

\begin{remark}
If $\beta(U)=\beta_{0}>0$ and $q=d$, $\mathrm{tr}(Y)$ has a one-dimensional
variance Gamma distribution with the following characteristic function: for
$\theta\in\mathbb{R}$, $\Theta=\theta\mathrm{I}_{d},$%
\begin{align*}
\mathbb{E}\exp(i\theta\mathrm{tr}(Y))  &  =\exp\left\{  -\int_{\mathbf{S}%
_{\left\Vert \cdot\right\Vert }^{+}}\ln\left(  1+\frac{1}{2\beta_{0}}%
\theta^{2}\right)  \alpha(\mathrm{d}U)\right\} \\
&  =\left(  1+\frac{1}{2\beta_{0}}\theta^{2}\right)  ^{-\alpha(\mathbf{S}%
_{\left\Vert \cdot\right\Vert }^{+})}\\
&  =\left(  1-i\frac{1}{\sqrt{2\beta_{0}}}\theta\right)  ^{-\alpha
(\mathbf{S}_{\left\Vert \cdot\right\Vert }^{+})}\left(  1+i\frac{1}%
{\sqrt{2\beta_{0}}}\theta\right)  ^{-\alpha(\mathbf{S}_{\left\Vert
\cdot\right\Vert }^{+})}..
\end{align*}
Thus, $\mathrm{tr}(Y)$ has the same distribution as $V-V%
\acute{}%
$, where $V$ has a one-dimensional Gamma distribution $\Gamma(\alpha
(\mathbf{S}_{\left\Vert \cdot\right\Vert }^{+}),\sqrt{2\beta_{0}})$ and $V%
\acute{}%
$ is an independent copy of $V.$
\end{remark}

\begin{remark}
Let $q=1$, i.e. the resulting matrix Gamma-normal distribution is $\bbr^d$-valued. Thus for $\Theta\in\mathbb{R}^{d}$, $\Theta\Theta^{\top}$ has rank
one. 

Assume additionally that the measure $\alpha$ is concentrated on the  rank one matrices,
that is $U=%
\operatorname{u}%
\operatorname{u}%
^{\top}$ with $\operatorname{u}\in\bbr^d$ (and the first non-zero component of $\operatorname{u}$ being positive, to make $\operatorname{u}$ unique given $U$). Let $\widetilde{\alpha}$ be the measure on the unit sphere
$\mathbf{S}_{\mathbb{R}^{d},\Vert\cdot\Vert}$ of $\mathbb{R}^{d}$ induced by
$\alpha$ under this transformation. Using this we write the integral in the
right hand side of (\ref{ChFMatGGam}) as follows%
\begin{align}
&  \int_{\mathbf{S}_{\left\Vert \cdot\right\Vert }^{+}}\ln\left(  1+\frac
{1}{2}\frac{\mathrm{tr}(U\Theta\Theta^{\top})}{\beta(U)}\right)
\alpha(\mathrm{d}U)\nonumber\\
&  =\int_{\mathbf{S}_{\mathbb{R}^{d},\Vert\cdot\Vert}}\ln\left(  1+\frac{1}%
{2}\frac{\left(  \Theta^{\top}%
\operatorname{u}%
\right)  ^{2})}{\beta(%
\operatorname{u}%
\operatorname{u}%
^{\top})}\right)  \widetilde{\alpha}(\mathrm{d}%
\operatorname{u}%
)\nonumber\\
&  =\int_{\mathbf{S}_{\mathbb{R}^{d},\Vert\cdot\Vert}}\ln\left(
1-i\frac{\Theta^{\top}%
\operatorname{u}%
}{\sqrt{2\beta(%
\operatorname{u}%
\operatorname{u}%
^{\top})}}\right)  \widetilde{\alpha}(\mathrm{d}%
\operatorname{u}%
)\nonumber\\
&  \qquad+\int_{\mathbf{S}_{\mathbb{R}^{d},\Vert\cdot\Vert}}\ln\left(
1+i\frac{\Theta^{\top}%
\operatorname{u}%
}{\sqrt{2\beta(%
\operatorname{u}%
\operatorname{u}%
^{\top})}}\right)  \widetilde{\alpha}(\mathrm{d}%
\operatorname{u}%
). \label{tem}%
\end{align}
Interestingly, (\ref{tem}) implies that the matrix Gamma-normal random variable can be represented (in this special case) as $X_1-X_2$ with $X_1,X_2\sim \Gamma_d(\tilde \alpha,\tilde \beta)$ being independent where $\tilde \beta=\sqrt{2\beta(\operatorname{uu}^\top)}$.

Hence, the matrix Gamma-normal distribution with $q=1$, which can indeed be regarded as a $d$-dimensional generalisation of the univariate variance Gamma distribution, inherits interesting properties well-known in the univariate case.  
\end{remark}

 \ACKNO{This work was partially supported by the Technische
Universität München -- Institute for Advanced Study funded by the German
Excellence Initiative by a Visting Fellowship for VPA and a Carl-von-Linde Junior Fellowship for RS.}

\end{document}